\documentclass[a4paper,reqno]{amsart}

\usepackage{amscd,amsthm,amsmath,amssymb,mathtools}
\usepackage{verbatim}
\usepackage{color}
\usepackage{url}
\usepackage{enumerate,enumitem}
\usepackage{graphicx}
\usepackage{epstopdf,epsfig,subfigure}
\usepackage{curve2e}
\usepackage{array}
\usepackage{stackrel}
\usepackage{multirow}

\usepackage[numbers,sort&compress]{natbib}

\usepackage{hyperref}

\theoremstyle{plain}

\newtheorem{theorem}{Theorem}[section]

\newtheorem{lemma}[theorem]{Lemma}
\newtheorem{corollary}[theorem]{Corollary}

\newtheorem{assumption}[theorem]{Assumption}

\theoremstyle{definition}
\newtheorem{definition}[theorem]{Definition}
\newtheorem{remark}[theorem]{Remark}

\numberwithin{equation}{section}

\usepackage{geometry}

\newcommand{\rest}{\left.\kern-2\nulldelimiterspace\right|_}

\newcommand{\rank}{\mathop{\rm rank}\nolimits}

\newcommand{\Id}{{\mathbf1}}

\newcommand{\p}{\partial}

\newcommand{\clA}{{\mathcal A}}

\newcommand{\clE}{{\mathcal E}}

\newcommand{\clH}{{\mathcal H}}
\newcommand{\clI}{{\mathcal I}}
\newcommand{\clJ}{{\mathcal J}}

\newcommand{\clU}{{\mathcal U}}
\newcommand{\clV}{{\mathcal V}}

\newcommand{\bbC}{{\mathbb C}}

\newcommand{\bbN}{{\mathbb N}}

\newcommand{\bbR}{{\mathbb R}}

\newcommand{\bbZ}{{\mathbb Z}}

\newcommand{\bfA}{{\mathbf A}}
\newcommand{\bfB}{{\mathbf B}}

\newcommand{\bfD}{{\mathbf D}}

\newcommand{\bfG}{{\mathbf G}}

\newcommand{\bfK}{{\mathbf K}}

\newcommand{\bfZ}{{\mathbf Z}}

\newcommand{\fkC}{{\mathfrak C}}

\newcommand{\bfe}{{\mathbf e}}

\newcommand{\bfg}{{\mathbf g}}

\newcommand{\bfo}{{\mathbf o}}
\newcommand{\bfp}{{\mathbf p}}

\newcommand{\bfx}{{\mathbf x}}
\newcommand{\bfy}{{\mathbf y}}

\newcommand{\rmd}{{\mathrm d}}
\newcommand{\rme}{{\mathrm e}}

\definecolor{DarkBlue}{rgb}{0,0.08,0.45}
\definecolor{DarkRed}{rgb}{.65,0,0}
\definecolor{applegreen}{rgb}{0.55, 0.71, 0.0}

\newcounter{mymac@matlab}
\newcommand{\matlab}{MATLAB%
   \ifnum\value{mymac@matlab}<1%
   \textregistered%
   \setcounter{mymac@matlab}{1}%
   \fi%
  }

\newcommand{\bfclE}{{\bf {\mathcal{E}}}}
\newcommand{\bfPi}{{\mathbf \Pi}}

\begin{document}
\title{Ensemble Feedback Stabilization of Linear Systems}
\author{Philipp A.~Guth$^{\tt1}$, Karl Kunisch$^{\tt1,2}$, and S\'ergio S.~Rodrigues$^{\tt1}$}
\thanks{
\vspace{-1em}\newline\noindent
{\sc MSC2020}: 34H05, 49J15, 49N10, 93B52, 34F05.
\newline\noindent
{\sc Keywords}: Ensemble of parameter-dependent systems, Feedback control, Algebraic Riccati equation, Ensemble stabilization
\newline\noindent
$^{\tt1}$ Johann Radon Institute for Computational and Applied Mathematics,
  \"OAW, 
  Altenbergerstrasse~69, 4040~Linz, Austria.\newline\noindent
$^{\tt2}$   Institute of Mathematics and Scientific Computing, Karl-Franzens University of Graz,	     	
Heinrichstrasse~36, 8010 Graz, Austria.\newline\noindent
{\sc Emails}:
{\small\tt   philipp.guth@ricam.oeaw.ac.at,\quad karl.kunisch@uni-graz.at,\quad\\ \hspace*{3.4em}sergio.rodrigues@ricam.oeaw.ac.at}
 }

\begin{abstract}
Stabilization of linear control systems with parameter-dependent system matrices is investigated. A Riccati based feedback mechanism is proposed and analyzed. It is constructed by means of an ensemble of parameters from a training set. This single  feedback stabilizes all systems of the training set and also systems in its vicinity. Moreover its suboptimality with respect to optimal feedback for each single parameter from the training set can be quantified.
\end{abstract}

\maketitle

\pagestyle{myheadings} \thispagestyle{plain} \markboth{\sc P.A. Guth, K. Kunisch, and S.S. Rodrigues}{\sc Ensemble Feedback Stabilization of Linear Systems}

\section{Introduction}
Stabilization of dynamical systems using feedback control is an important task in science and engineering problems. In real-world applications the involved system equations  frequently dependent on uncertain or even unknown parameters. The stabilization of these systems can be a challenging task since already small changes in the parameters may change the stability properties of the uncontrolled system. In this work we develop a feedback control which stabilizes a class of parameter-dependent linear systems for each realization of the parameter.

This problem falls into the larger class of optimization under uncertainty, see, for example, \cite{azmi2023analysis,guth2022parabolic,kunoth2013analytic,martinez2016robust}. From the control point perspective, these contributions treat open-loop optimal control problems or stationary optimization problems. In our work we focus on optimal control problems in feedback form. They are posed on the infinite time horizon. Thus they represent the optimal control formulation of stabilization problems.

We point out that the parameters in our work enter the model through the system matrix, and thus the problem under investigation differs from treatment of stochastic optimal control problems where the noise enters in an affine manner, as for instance in \cite[Ch.~3.6]{kwakernaak1969linear} or \cite[Ch.~III]{fleming2006controlled}.

It appears to be the case that feedback under uncertainty in the coefficients has received rather little attention in the literature so far, and thus we consider our work as one possible step in this direction. Certainly other approaches are conceivable and their analysis can be of interest in future work.

The manuscript is structured as follows. In Section~\ref{sec:Encontrstab} the notion of ensemble stabilizability is introduced and necessary and sufficient conditions are derived for it to hold. This notion is intimately related to ensemble controllability which is well-known from the literature. A robust linear feedback is proposed and analyzed in Section~\ref{sec:linfb}. In Section~\ref{sec:enstabsys} we verify the applicability of our theoretical results to some models of real-world phenomena. Results of numerical experiments are reported in Section~\ref{sec:numEx}.

\subsection{Related Literature}
Controlled systems with uncertainties entering the system matrix arise, for instance, from control problems involving physical models with uncertain or unknown parameters and hence they are relevant in various fields. Examples include compartmental models with uncertain coefficients, oscillatory systems with uncertain damping coefficients, and spatial discretization of controlled partial differential equations.

For (spatial discretizations of) parabolic equations, in the context of open-loop control, this problem class has been studied recently for both finite time horizon \cite{guth2022parabolic,kunoth2013analytic,martinez2016robust} and infinite time horizon \cite{azmi2023analysis}. Thereby the input randomness is typically expressed in terms of a series expansion (e.g., Karhunen--Lo\`eve expansion) and then approximated by truncating after finitely many terms (see, e.g.,~\cite{nobile2009analysis,zhang2012error}). To account for the stochastic response of the system state, the cost functional of an optimal control problem needs to be composed with a risk measure, such as the expected value. For the numerical approximation of the risk measures, which typically involve high-dimensional (the dimension is given by the order of truncation of the series expansion) integrals of the system output, cubature rules, such as Monte Carlo or quasi-Monte Carlo methods are used.

Research towards feedback controls for parameterized systems include the following. In \cite{yedavalli2014robust} robustness criteria for linear systems are investigated. For instance, error bounds are obtained for the perturbed system and control matrices under which a Riccati based nominal feedback law remains stable. In \cite{kramer2017feedback} the authors propose an online-offline strategy to stabilize a parameter dependent controlled dynamical system. In the offline phase a field of stabilizing feedbacks is precomputed for sampled parameter values. These are used in the online phase during which a classification is carried out to determine the current (time-dependent) parameter value.

The concept of ensemble controllability, that is, the controllability of ensembles of systems, which is also referred to as simultaneous controllability, is investigated in \cite{LAZAR2022265}, \cite[Ch.~5]{lions1988controlabilite}, \cite[Ch.~11.3]{tucsnak2009observation}. Related concepts are the notions of uniform ensemble controllability and $L^p$-ensemble controllability \cite{danhane2022conditions,helmke2014uniform}. In \cite{zuazua2014averaged} the notion of avaraged controllabity, is discussed and a Kalman-type rank condition is derived. Results on averaged controllability for time-dependent PDEs with uncertain coefficients can be found in \cite{coulson2019average} and the references therein.

Our theoretical analysis is based on the notion of ensemble stabilizability. Among works concerned with stabilization of ensembles of systems using feedback controls, we find~\cite{chittaro2018asymptotic}, where stabilizability is investigated for an ensemble of Bloch equations, and~\cite{ryan2014simultaneous}, where a globally asymptotic bilinear stabilizing feedback for an ensemble of oscillators with pairwise distinct free-dynamics frequencies is developed.


\section{Ensemble Controllability and Ensemble Stabilizability}\label{sec:Encontrstab}
We address the design of a feedback control which can be used to effectively steer each member of an ensemble of linear systems. Such an ensemble arises, for example, from the parameterization of uncertain coefficients in the system. This feedback will be introduced in Section~\ref{sec:linfb}. In this section we introduce and discuss the notions of ensemble controllability and ensemble stabilizability. We commence from a given finite ensemble (sequence) of parameters $\varSigma\coloneqq (\sigma_i)_{i=1}^N$, where~$N\ge2$ is an integer, and consider the ensemble of dynamical linear control systems, for time~$t>0$, as follows
\begin{align}\label{eq:sys}
 \dot{x}_{\sigma_i} &= \clA_{\sigma_i} x_{\sigma_i} + Bu_{\sigma_i},\qquad  x_{\sigma_i}(0) = x_{\circ}, &&1\le i\le N, 
\end{align}
where  $\dot{x}_{\sigma_i}\coloneqq\frac{\rmd}{\rmd t}{x}_{\sigma_i}$, $\clA_{\sigma_i} \in \bbR^{n \times n}$, $x_{\sigma_i}(t) \in \bbR^{n}$ and~$u_{\sigma_i}(t) \in \bbR^{m}$ for all $i=1,\ldots,N$, $B\in \bbR^{n \times m}$,  and $x_{\circ} \in \bbR^{n}$ is a given initial condition. The subscript $\sigma_i$ denotes the dependence of~$\clA$, $x$, and~$u$ on the $i$--th parameter.

We shall derive a feedback control for~\eqref{eq:sys}, where the extended control system
 \begin{align}\label{eq:extsys}
  \dot{\bfx} &= \bfA_{\varSigma} \bfx + \bfB u,\qquad  \bfx(0) = \bfx_{\circ},
 \end{align}
is used as an auxiliary system, with initial condition $\bfx_{\circ} \in \bbR^{nN}$ and block matrices
 \begin{align}\label{eq:extsys-op}
 \bfA_{\varSigma}&\coloneqq\begin{bmatrix}\clA_{\sigma_1}& \\
 & \clA_{\sigma_2}\\
 & &\ddots\\
 & & & \clA_{\sigma_N}\end{bmatrix} \in \bbR^{nN\times nN},\quad \mbox{and} \quad
 \bfB\coloneqq\begin{bmatrix}B \\B\\ \vdots\\B\end{bmatrix} \in \bbR^{nN\times m}.
 \end{align}
\begin{remark}
To make the connection with~\eqref{eq:sys}, later on, we shall be particularly interested in initial conditions as~$\bfx_\circ=\begin{bmatrix} x_{\circ}^\top& x_{\circ}^\top& \hdots& x_{\circ}^\top \end{bmatrix}^\top \in \bbR^{nN\times 1}$  with~$x_\circ\in\bbR^{n\times 1}$.
\end{remark}

We shall look for input controls (as~$u_{\sigma_i}$ and~$u$ above) which are in~$L^2((0,+\infty);\bbR^m)$.
\begin{definition}
Given a finite ensemble of parameters $\varSigma=  (\sigma_i)_{i=1}^N$, we say that~$(\bfA_{\varSigma},\bfB)$ (or system~\eqref{eq:extsys}) is controllable if for each pair of states~$\bfx_{\circ} \in \bbR^{nN}$ and~$\bfx_1 \in \bbR^{nN}$, there exists~$T>0$ and a control $u \in L^2((0,T);\bbR^m)$ so that the solution $\bfx$ of~\eqref{eq:extsys} satisfies~$\bfx(T)=\bfx_1$.
\end{definition}

\begin{definition}
 For a given finite ensemble of parameters $\varSigma=  (\sigma_i)_{i=1}^N$, we say that the ensemble of systems~$(\clA_{\sigma_i},B)_{i=1}^N$ is \emph{ensemble controllable} if ~$(\bfA_{\varSigma},\bfB)$ is controllable.
\end{definition}

By introducing the matrix
\begin{equation}\notag
\begin{bmatrix}\bfA_{\varSigma}\,\colon\bfB\end{bmatrix}\coloneqq\begin{bmatrix}\bfB&\bfA_{\varSigma}\bfB&\ldots&\bfA_{\varSigma}^{nN-1}\bfB\end{bmatrix}\in\bbR^{nN\times mnN},
\end{equation}
the ensemble controllability of $(\clA_{\sigma_i},B)_{i=1}^N$ can be verified using the Kalman rank condition~\cite[Cor.~1.4.10]{tucsnak2009observation}: $(\bfA_{\varSigma},\bfB)$ is controllable iff (if and only if)
\begin{align}\label{eq:Kalman}
 \rank \begin{bmatrix}\bfA_{\varSigma}\,\colon\bfB\end{bmatrix} = nN.
\end{align}
Alternatively, denoting the set of the eigenvalues of~$A\in\bbR^{n\times n}$ by~${\rm Eig}(A)\subset\bbC$, we can use the Hautus test \cite[Thms.~1 and~1$^\prime$]{HautusContr}, \cite[Thm.~1]{HautusStab}:  $(\bfA_{\varSigma},\bfB)$ is controllable iff
\begin{align}\label{eq:HautusContr}
 \rank\begin{bmatrix}\bfA_{\varSigma}-\lambda\Id_{nN}&\bfB\end{bmatrix} = nN\quad\mbox{for all}\quad\lambda \in {\rm Eig}(\bfA_{\varSigma}).
\end{align}

\begin{definition}
Given a finite ensemble of parameters $\varSigma=  (\sigma_i)_{i=1}^N$, we say that~$(\bfA_{\varSigma},\bfB)$ (i.e., that system~\eqref{eq:extsys}) is stabilizable if there exists a matrix $\bfK_{\varSigma} \in \bbR^{m \times nN}$ such that $\bfA_{\varSigma} + \bfB \bfK_{\varSigma}$ is stable.
\end{definition}

\begin{definition}
 For a given finite ensemble of parameters $\varSigma=  (\sigma_i)_{i=1}^N$, we say that the ensemble $(\clA_{\sigma_i},B)_{i=1}^N$ is \emph{ensemble stabilizable} if $(\bfA_{\varSigma},\bfB)$ is stabilizable.
\end{definition}

Ensemble stabilizability of $(\clA_{\sigma_i},B)_{i=1}^N$ can be verified by the Hautus test (for stabilizability) \cite[Thm.~4]{HautusStab}: $(\bfA_{\varSigma},\bfB)$ is stabilizable iff
\begin{align}\label{eq:HautusStab}
 \rank\begin{bmatrix}\bfA_{\varSigma}-\lambda\Id_{nN}&\bfB\end{bmatrix} = nN\quad\mbox{for all}\quad\lambda \in {\rm Eig}_{\ge0}(\bfA_{\varSigma}),
\end{align}
where~${\rm Eig}_{\ge 0}(\bfA_{\varSigma}) \coloneqq \{\lambda \in {\rm Eig}(\bfA_{\varSigma})\mid \mathfrak{Re}(\lambda) \geq 0\}$ and~$\mathfrak{Re}(\lambda)$ is the real part of~$\lambda\in\bbC$.

By \eqref{eq:HautusContr} and \eqref{eq:HautusStab}, ensemble controllability implies ensemble stabilizability.

Now, we introduce the following notation:
by $\langle x, y \rangle$, with~$x,y \in \bbR^M$, we denote the Euclidean scalar product on $\bbR^M$ and its associated norm by $\|x\| \coloneqq \sqrt{\langle x, x \rangle}$. For a matrix $A\in \bbR^{n\times m}$ we denote its linear operator norm by $\|A\| \coloneqq \sup\limits_{y \in \bbR^m \setminus \{0\}} \frac{\|Ay\|}{\|y\|}$. Finally, by $\Id_n$ we denote the identity matrix in $\bbR^{n \times n}$.

Next, note that given~$\alpha > 0$, we have that $(\bfA_{\varSigma},\bfB)$ is stabilizable iff $(\bfA_{\varSigma},\frac{1}{\sqrt\alpha}\bfB)$ is stabilizable. Thus, it is well known that for every $\alpha > 0$ there exists a unique solution~$\bfPi_{\varSigma}\succ0$ (i.e., positive definite) of the algebraic Riccati equation
\begin{align}\label{eq:bigRiccati}
 \bfA_{\varSigma}^\top{\bfPi_{\varSigma} + \bfPi_{\varSigma} \bfA_{\varSigma} }- \frac{1}{\alpha}{\bfPi_{\varSigma} \bfB \bfB^\top \bfPi_{\varSigma}} + \frac{1}{ N}\Id_{nN} = 0,
\end{align}
and that the matrix ${\bfA_{\varSigma}} - \frac{1}{\alpha} \bfB \bfB^\top \bfPi_{\varSigma}$ is stable (see, e.g.,~\cite[Thm.~9.5]{zabczyk2020controltheory}).
Moreover, the matrix $\bfPi_{\varSigma}$ provides the optimal control $u_{\varSigma}$ in feedback form
\begin{align}\label{eq:extcontr}
u_{\varSigma}(t) = -\frac{1}{\alpha}  \bfB^{\top} \bfPi_{\varSigma} \bfx(t),\qquad t\geq 0,
\end{align}
which minimizes the functional
\begin{align}\label{eq:extobj}
\clJ(\bfx,u) = \frac12 \int_0^{+\infty} \Big( \frac{1}{ N} \|\bfx(s)\|^2 + \alpha \|u(s)\|^2 \Big)\rm ds
\end{align}
subject to~\eqref{eq:extsys}. The minimal value of $\clJ$ equals $\frac12 \langle\bfx_{\circ}, \bfPi_{\varSigma} \bfx_{\circ}\rangle$ (see, e.g.,~\cite[Thm.~9.4]{zabczyk2020controltheory}).

We note that the closed-loop control $u_{\varSigma}$ given in \eqref{eq:extcontr} coincides with the open-loop optimal control obtained from the first order optimality relations associated with the same optimization problem, namely,
\begin{subequations}\label{eq:optcondext}
 \begin{align}
\dot{\bfx} &= \bfA_{\varSigma} \bfx+ \bfB u,\qquad\bfx(0) = \bfx_{\circ},\label{eq:optcondextA}\\
\dot{\bfp}&= -\bfA_{\varSigma}^\top \bfp - \frac{1}{ N}\bfx,\label{eq:optcondextC}\\
u &= -\frac{1}{\alpha}\bfB^\top \bfp.\label{eq:optcondextE}
 \end{align}
 \end{subequations}
Further, we recall that~\eqref{eq:extcontr} follows by the dynamic programming principle, from which we obtain $\bfp(t) = \bfPi_{\varSigma}\bfx(t)$ for $t\ge 0$. Finally, we see that
 \begin{align}\label{eq:optcondextxpinW}
\bfx\in W((0,+\infty);\bbR^{nN})\quad\mbox{and}\quad\bfp\in W((0,+\infty);\bbR^{nN}),
 \end{align}
where we denote $W(I;\bbR^{M})\coloneqq\{z\in L^2(I;\bbR^{M})\mid \dot z\in L^2(I;\bbR^{M})\}$ defined in a time interval~$I\subseteq(0,+\infty)$, for a positive integer~$M$.

 In particular, the stability of $\bfA_{\varSigma}+\bfB\bfK_\varSigma$, with~$\bfK_\varSigma=-\frac{1}{\alpha}\bfB^\top\bfPi_{\varSigma}$, gives us that $\lim_{t \to \infty} \bfx(t) =0$ and consequently~$\lim_{t \to \infty} \bfp(t) =0$.

\subsection{Relation between Stabilizability and Ensemble Stabilizability}\label{sec:Rel}

We start with a result providing necessary and sufficient conditions for ensemble stabilizability. Results analogous to the following lemma are known for controllability and can be found in \cite[Prop. 4.1]{danhane2022conditions} and \cite[Lem. 1]{helmke2014uniform}.

\begin{lemma}\label{lem:stab-enstab}
Let the ensemble of systems $(\clA_{\sigma_i},B)_{i=1}^N$ be ensemble stabilizable. Then, for every $1\le i\le N$, the system $(\clA_{\sigma_i},B)$ is stabilizable, and there holds
\begin{align}
\bigcup_{\substack{\clI \subseteq\{1,\ldots,N\}\\ \#\clI =m+1}} \bigcap_{j\in\clI} {\rm Eig}_{\ge 0}(\clA_{\sigma_j}) = \emptyset,\label{CupCapEigs}
\end{align}
where $m$ denotes the dimension of the control space and $\#\clI$ the cardinality of $\clI$.
Reciprocally, let the system $(\clA_{\sigma_i},B)$ be stabilizable for each $1\le i\le N$, and let
\begin{align}
 {\rm Eig}_{\ge 0}(\clA_{\sigma_i}) \cap {\rm Eig}_{\ge 0}(\clA_{\sigma_j}) = \emptyset\quad \mbox{for all } i \ne j, \quad1\le i,j\le N.\notag
\end{align}
Then, the ensemble of systems $(\clA_{\sigma_i},B)_{i=1}^N$ is ensemble stabilizable.
\end{lemma}
\begin{proof}
 To prove the first statement, suppose that the $i$-th system $(\clA_{\sigma_i},B)$, $1\le i \le N$ is not stabilizable. Then, the Hautus test implies that there exists a vector $v_i\in\bbC^{n\times1}\setminus\{0\}$ such that $v_i^\top B = 0\in\bbC^{1\times m}$ and $v_i^\top(\clA_{\sigma_i} - \lambda \Id_n) = 0$ for an eigenvalue $\lambda\in{\rm Eig}_{\ge0}(\clA_{\sigma_i})$. For $v^\top \coloneqq \begin{bmatrix} 0&\ldots&0&v_i^\top&0&\ldots &0 \end{bmatrix}\in\bbC^{1\times nN}$ there holds $0 = v_i^\top (\clA_{\sigma_i}-\lambda \Id_n) = v^\top (\bfA_{\varSigma} - \lambda \Id_{nN})$ and $0=v_i^\top B=v^\top \bfB$, thus $(\bfA_{\varSigma},\bfB)$ is not stabilizable, by the Hautus test \eqref{eq:HautusStab}. 
 We conclude that the ensemble stabilizability of $(\clA_{\sigma_i},B)_{i=1}^N$ holds only if we have the stabilizability of every system $(\clA_{\sigma_i},B)$, $1\le i \le N$.

Suppose now that there exists a subset~$\clI \subseteq\{1,\ldots,N\}$ satisfying~$\#\clI = m+1$ and~$\Lambda\coloneqq\bigcap_{j\in\clI} {\rm Eig_{\ge 0}}(\clA_{\sigma_j}) \neq \emptyset$. Without loss of generality (up to a reordering of~$\varSigma$) we can assume that~$\clI=\{1,2,\dots,m+1\}$. Since ${\rm Eig}(\clA_{\sigma_i})={\rm Eig}(\clA_{\sigma_i}^\top)$, by taking~$\lambda\in\Lambda$ we know that there are~$v_i\in\bbC^{n\times1}\setminus\{0\}$ (eigenvector of~$\clA_{\sigma_i}^\top$) satisfying $v_i^\top \clA_{\sigma_i} = \lambda v_i^\top$ for each $1\le i\le m+1$. Since~$m+1 >m$ we can find a vector $\alpha \in \bbC^{m+1}\setminus\{0\}$ such that $\sum_{i=1}^{m+1} \alpha_i v_i^\top B = 0$. Setting $\tilde{v}^\top \coloneqq \begin{bmatrix} \alpha_1 v_1^\top & \alpha_2 v_2^\top &\ldots & \alpha_{m+1} v_{m+1}^\top & 0 &\ldots & 0\end{bmatrix}\ne0$ gives $\tilde{v}^\top \begin{bmatrix} \bfA_{\varSigma} - \lambda \Id_{nN} & \bfB \end{bmatrix} = 0$. Since~$\lambda\in{\rm Eig}_{\ge0}(\bfA_{\varSigma})$ we can conclude, by the Hautus test, that~$(\bfA_{\varSigma},\bfB)$ is not stabilizable. This ends the proof of the first statement.

To prove the second statement, we suppose that $(\bfA_{\varSigma},\bfB)$ is not stabilizable. Then, by the Hautus test, there are $\lambda\in{\rm Eig}_{\ge0}(\bfA_{\varSigma})$ and~$v\in\bbC^{nN\times1}\setminus\{0\}$ such that $v^\top\bfB = 0$ and $v^\top(\bfA_{\varSigma}-\lambda\Id_{nN}) = 0$. In particular, there holds $v_i \ne 0$ for at least one $i\in\{1,2,\dots, N\}$, where $v^\top = \begin{bmatrix} v_1^\top&\ldots&v_i^\top&\ldots& v_N^\top \end{bmatrix}$, with~$ v_j\in\bbC^{n\times1}$, $1\le j\le N$. Next, since $v^\top \bfA_{\varSigma} = \lambda v^\top$, we have that if $\mbox{\rm Eig}_{\ge 0}(\clA_{\sigma_j})\, \cap\, \mbox{\rm Eig}_{\ge 0}(\clA_{\sigma_i}) = \emptyset$ for all $i \ne j$, then there can be at most one $v_i \ne 0$. In this case, $v^\top = \begin{bmatrix} 0&\ldots&0&v_i^\top&0&\ldots &0 \end{bmatrix}$ and thus $0 = v^\top\bfB = v_i^\top B$ and $0 = v^\top(\bfA_{\varSigma} - \lambda \Id_{nN}) = v_i^\top(\bfA_{\sigma_i} - \lambda \Id_n)$, which implies that the~$i$-th system~$(\clA_{\sigma_i},B)$ not stabilizable, by the Hautus test. This ends the proof of the second statement. \end{proof}


\section{A Linear Feedback}\label{sec:linfb}
In this section we return to the construction of  a feedback law for an ensemble of systems. For a given control operator $B\in \bbR^{n \times m}$ and a given initial condition $x_{\circ} \in \bbR^{n}$, we want to find a control $u_{\sigma} \in \bbR^{m}$ in feedback form, i.e., $u_{\sigma} = Kx_{\sigma}$ for some $K\in \bbR^{m\times n}$ independent of $\sigma$, which stabilizes the system
\begin{align}\notag
 \dot{x}_{\sigma} &= \clA_{\sigma} x_{\sigma}+ Bu_{\sigma},\qquad x_{\sigma}(0) = x_{\circ},
\end{align}
with a matrix~$\clA_{\sigma}\in \bbR^{n \times n}$ depending on an unknown parameter $\sigma\in\bbR$. This can be a challenging task since already small changes in~$\sigma$ may change the stability properties of the uncontrolled system. It is therefore important to take into account the uncertainty when constructing a feedback for the system, especially in situations in which it is prohibitively costly or impossible to obtain an a-priori estimate for~$\sigma$.

We construct a feedback based on a finite ensemble $\varSigma = (\sigma_i)_{i=1}^N$, $N\ge2$, of possible values for~$\sigma$ and derive conditions which ensure that this feedback stabilizes the system for the ``true''~$\sigma$.
Identifying~$\bbR^M$ with~$\bbR^{M\times1}$ we define the extension operator
\begin{align}\label{eq:extOp}
\clE: \bbR^n \to \bbR^{nN},\qquad
x \mapsto\begin{bmatrix} x^\top & x^\top &\ldots & x^\top \end{bmatrix}^\top,
\end{align}
and introduce the linear feedback control operator
\begin{align}\label{eq:linFB}
-\frac{1}{\alpha}\bfB^\top \bfPi_{\varSigma} \clE \in \bbR^{m\times n},
\end{align}
where~$\bfB$ is as in~\eqref{eq:extsys-op} and~$\bfPi_{\varSigma}\succ0$ solves~\eqref{eq:bigRiccati}. It leads to the feedback control input
\begin{align}\label{eq:fbcontrol}
u_{\varSigma,\sigma} \coloneqq -  \frac{1}{\alpha} \bfB^\top \bfPi_{\varSigma} \clE x_{\varSigma,\sigma},
\end{align}
and the associated closed-loop system
\begin{align}
\dot{x}_{\varSigma,\sigma} &= \clA_{\sigma} x_{\varSigma,\sigma} - \frac{1}{\alpha}B \bfB^\top \bfPi_{\varSigma} \clE x_{\varSigma,\sigma},\qquad
x_{\varSigma,\sigma}(0) = x_{\circ}.\label{eq:avclsys}
\end{align}
Here $\sigma$ is not necessarily an element of $\varSigma$.
In this framework we can consider the parameters $\sigma_i$ as training parameters which serve for the construction of $\bfPi_\Sigma$. We have tested the proposed feedback law \eqref{eq:fbcontrol} for different application-motivated situations and shall  report on numerical results in Section~\ref{sec:numEx}. The results are very promising both in situations where the parameters $\sigma_i$ correspond to stable and unstable matrices~$\clA_{\sigma_i}$. The feedback law also provides good results in some cases when applied to systems~$(\clA_{\sigma},B)$ with $\sigma$ not in the convex hull of the training parameters $\sigma_i\in\varSigma$. These results motivate the analysis of some of its structural properties.
For this purpose, we make the following assumption throughout the rest of the paper.
\begin{assumption}\label{assump1}
 For the finite ensemble of parameters $\varSigma = (\sigma_i)_{i=1}^N$, the ensemble of systems~$(\clA_{\sigma_1},B)_{i=1}^N$ as in~\eqref{eq:sys} is ensemble stabilizable.
\end{assumption}
In Lemma~\ref{lem:stab-enstab} we provided necessary and sufficient conditions for this assumption to hold.
In particular, we have that $\bigcup_{\substack{\clI \subseteq\{1,\ldots,N\}\\ \#\clI = m+1}}\bigcap_{j\in\clI} {\rm Eig}_{\ge 0}(\clA_{\sigma_j}) = \emptyset$ if the Assumption~\ref{assump1} holds. Conversely, note that since $N>1$, if $\sigma_i=\sigma_1$ for all~$1\le i\le N$ and if~$\lambda\in{\rm Eig}_{\ge 0}(\clA_{\sigma_1})$, then the Assumption~\ref{assump1} cannot hold if~$m+1 \le N$.

Before we present the results for the feedback \eqref{eq:linFB}, in the following subsection we first compare the minimizer $(\bfx,u_{\varSigma})$ of $\clJ$ as defined in \eqref{eq:extobj} subject to the extended system \eqref{eq:extsys} to its counterpart for a single, fixed parameter $\sigma$: find a minimizing pair $(x_\sigma,u_{\sigma})$ of
\begin{align}\label{eq:singlecost}
 \clJ(\clE x_{\sigma},u) = \frac12 \int_0^{+\infty} \Big(\|x_{\sigma}(t)\|^2+ \alpha \|u(t)\|^2\Big){\rm dt}
\end{align}
subject to
\begin{align}
 \dot{x}_{\sigma} &= \clA_{\sigma} x_{\sigma}+ Bu,\qquad x_{\sigma}(0) = x_{\circ}, \label{eq:syssingle} 
\end{align}
Note that $\frac1 N \|\clE y\|^2 = \|y\|^2$ for $y \in \bbR^{n}$.

Next, we define $\|\bfy\|_{\bfPi_{\varSigma}} \coloneqq \sqrt{\langle \bfy,\bfPi_{\varSigma} \bfy \rangle}$, for vectors $\bfy \in \bbR^{nN}$. Since~$\bfPi_{\varSigma}\succ0$ is symmetric, we can choose positive constants $\beta_1,\beta_2$ such that
\begin{align}\label{eq:Pinorm}
 \beta_1 \|\bfy\| \le \|\bfy\|_{\bfPi_{\varSigma}} \le \beta_2 \|\bfy\|, \qquad \mbox{for all}~\bfy \in \bbR^{nN}.
\end{align}
Note that~\eqref{eq:Pinorm} holds with $\beta_1 \coloneqq \sqrt{(\min({\rm Eig}(\bfPi_{\varSigma})))}$ and $\beta_2 \coloneqq \sqrt{(\max({\rm Eig}(\bfPi_{\varSigma})))}$. Recalling~\eqref{eq:bigRiccati}, $\bfPi_{\varSigma}$ depends on the weight~$\alpha$ of the control cost and the number~$N=\#\varSigma$ of parameters. Hence, also $\beta_1$ and~$\beta_2$ depend on~$(\alpha,N)$.

In the following we will frequently use the notation $\bfA_{\sigma}$ to denote a block-diagonal matrix containing $N$ identical blocks with the same parameter $\sigma$, that is,
\begin{equation}\notag
\bfA_{\sigma} \coloneqq\begin{bmatrix}\clA_{\sigma}& \\ & \ddots\\ & & \clA_{\sigma}\end{bmatrix}.
\end{equation}
We observe that the stabilizability of $(\bfA_{\varSigma},\bfB)$ implies that there exists $M_s>0$ such that, for all $\bfg \in L^2((0,+\infty);\bbR^{nN}),$ the problem
\begin{align}
\dot{\bfy} = \bfA_{\varSigma} \bfy - \bfB \bfB^\top \bfPi_{\varSigma} \bfy + \bfg,\qquad 
\bfy(0) = 0,\label{eq:g}
\end{align}
has a unique solution $\bfy \in W( (0,+\infty);\bbR^{nN})$ satisfying
\begin{align}\label{eq:w-est}
\|\bfy\|_{L^2( (0,+\infty);\bbR^{nN})} \leq M_s\|\bfg\|_{L^2( (0,+\infty);\bbR^{nN})}.
\end{align}
This holds true since $\mathbf{G} \coloneqq \bfA_{\varSigma} - \bfB \bfB^\top \bfPi_{\varSigma}$ is stable, that is, ${\rm Eig_{\ge0}(\bfG)}=\emptyset$. Hence, $\rme^{\bfG t}$ is exponentially stable, and the claim  follows (e.g., from~\cite[Prop.~3.7, Part~II, Ch.~ 1]{bensoussan2007representation}).

\subsection{On Optimal Controls and Costs} We compare the minimizers of~\eqref{eq:singlecost} and of~\eqref{eq:extobj} subject to the dynamics of the single parameter system~\eqref{eq:syssingle} and to the dynamics of the extended multi-parameter system~\eqref{eq:extsys}, respectively.
\begin{lemma}\label{lem:big-small}
Let $(x_{\sigma},u_{\sigma})$ be the minimizer of \eqref{eq:singlecost} subject to~\eqref{eq:syssingle}, and let $(\bfx,u_{\varSigma})$ be the minimizer of \eqref{eq:extobj} subject to~\eqref{eq:extsys} with~$\bfx(0)=\clE x_{\sigma}(0)$. Then, there holds
\begin{align}
\clJ(\bfx-\clE x_{\sigma},u_{\varSigma} - u_{\sigma})
\le \|\bfA_{\varSigma} - \bfA_{\sigma}\|^2 \fkC_{M_s,x_{\sigma},p_{\sigma},\alpha,\beta_2,\bfB},\notag
\end{align}
where, with~$\clH_n\coloneqq L^2( (0,+\infty);\bbR^{n})$,
\begin{subequations}\label{eq:big-smallfkC}
\begin{align}
&\fkC_{M_s,x_{\sigma},p_{\sigma},\alpha,\beta_2,\bfB, N} \coloneqq M_s \| x_{\sigma}\|_{\clH_n}  \| p_{\sigma}\|_{\clH_n} +  (1+\sqrt{\alpha N})^2\fkC_{M_s,x_{\sigma},p_{\sigma},\beta_2,\bfB}^2,\label{eq:big-smallfkC1}\\
\mbox{with}\quad 
&\fkC_{M_s,x_{\sigma},p_{\sigma},\beta_2,\bfB}\coloneqq \max\{M_s\| x_{\sigma}\|_{\clH_n}+\| p_{\sigma}\|_{\clH_n}, M_s \| x_{\sigma}\|_{\clH_n} \beta_2^2 \|\bfB\|\},\label{eq:big-smallfkC0}
\end{align}
\end{subequations}
where~$\bfB$, $\beta_2$, and~$M_s$ are as in~\eqref{eq:extsys-op}, \eqref{eq:Pinorm}, and~\eqref{eq:w-est}.
\end{lemma}
\begin{proof}
The optimality conditions for the problem with a fixed parameter~$\sigma$ are
\begin{subequations}\label{eq:optcond}
 \begin{align}
\dot{x}_{\sigma} &= \clA_{\sigma} x_{\sigma} + B u_\sigma, \qquad x_{\sigma}(0) = x_{\circ},\label{eq:optcondA}\\
\dot{p}_{\sigma}&= -\clA_{\sigma}^\top p_{\sigma} - x_{\sigma},\\
u_\sigma &= -\frac{1}{\alpha}B^\top p_{\sigma},\label{eq:optcondE}
 \end{align}
  \end{subequations}
where~$x_\circ$ is a given initial state. Now, with~$K_\sigma\coloneqq-\frac1{\alpha}B^\top\Pi_\sigma$, where~$\Pi_{\sigma}\succ0$ solves
 \begin{align}\label{eq:smallRiccati}
  \clA_{\sigma}^\top\Pi_{\sigma} + \Pi_{\sigma}  \clA_{\sigma} - \frac{1}{\alpha}{\Pi_{\sigma} B B^\top \Pi_{\sigma}} + \Id_{n} = 0,
\end{align}
we arrive at the analogue  of~\eqref{eq:optcondextxpinW},
  \begin{align}\label{eq:optcondxpinW}
{x}_{\sigma}\in W((0,+\infty);\bbR^{n})\quad\mbox{and}\quad p_{\sigma}=K_\sigma x_\sigma\in W((0,+\infty);\bbR^{n}).
 \end{align}
Using the optimality conditions~\eqref{eq:optcondext} and \eqref{eq:optcond} and defining $\delta\bfx \coloneqq \bfx - \clE x_{\sigma}$, $\delta\bfp \coloneqq \bfp - \frac1 N \clE p_{\sigma}$, as well as $\delta u \coloneqq u_{\varSigma} - u_{\sigma}$, we obtain
\begin{subequations}\label{eq:optcondiff}
\begin{align}
\dot{\delta\bfx} &= \bfA_{\varSigma} \delta\bfx + (\bfA_{\varSigma} - \bfA_{\sigma}) \clE x_{\sigma} + \bfB \delta u, \qquad\delta\bfx(0) = 0,\label{eq:optcondiffA}\\
\dot{\delta\bfp} &= -\bfA^\top_{\varSigma} \delta\bfp - (\bfA_{\varSigma} - \bfA_{\sigma})^\top \frac1 N \clE p_{\sigma} - \frac1 N \delta \bfx, \label{eq:optcondiffC}\\
\delta u &= -\frac{1}{\alpha}\bfB^\top\delta\bfp.\label{eq:optcondiffE}
\end{align}
\end{subequations}
Recalling~\eqref{eq:optcondextxpinW} and~\eqref{eq:optcondxpinW} and denoting the Hilbert spaces~$\clH\coloneqq L^2( (0,+\infty);\bbR^{nN})$ and~$\clU\coloneqq L^2( (0,+\infty);\bbR^{m})$, we find the following identities, with~$\delta\bfA\coloneqq\bfA_{\varSigma} - \bfA_{\sigma}$, 
\begin{align}
&-\langle \dot{\delta\bfp}, \delta\bfx \rangle_{\clH} = \langle  \bfA_{\varSigma}^\top \delta\bfp, \delta\bfx \rangle_{\clH}+ \langle \delta\bfA^\top \frac1 N \clE p_{\sigma}, \delta\bfx \rangle_{\clH} + \frac1 N \|\delta\bfx\|^2_{\clH} \label{eq:diffA}
\intertext{and, using~$\langle\delta\bfp(0), \delta\bfx(0) \rangle_{\clH}=0$ and~$\lim_{t\to+\infty}\langle\delta\bfp(t), \delta\bfx(t) \rangle_{\clH}=0$, we also have}
&-\langle\delta\bfx ,\dot{\delta\bfp}\rangle_{\clH} =\langle\dot{\delta\bfx} ,\delta\bfp\rangle_{\clH}= \langle \bfA_{\varSigma}\delta\bfx,\delta\bfp \rangle_{\clH}+ \langle \delta\bfA\clE x_{\sigma}, \delta\bfp \rangle_{\clH} + \langle \bfB\delta u,\delta\bfp \rangle_{\clH}.\label{eq:diffB}
\end{align}
Subtracting~\eqref{eq:diffB} from~\eqref{eq:diffA}, and using~\eqref{eq:optcondiffE}, lead us to
\begin{align}
\delta\clJ\coloneqq\frac1 N \|\delta\bfx\|^2_{\clH} +  \alpha \| \delta u\|^2_{\clU}&= -\langle \delta\bfA^\top \frac1 N \clE p_{\sigma}, \delta\bfx \rangle_{\clH}+ \langle \delta\bfA\clE x_{\sigma}, \delta\bfp \rangle_{\clH}\notag\\
&\le \|\delta\bfA\|\Bigl( \|\clE x_{\sigma}\|_{\clH}\,\|\delta\bfp\|_{\clH} + \frac1 N \|\clE p_{\sigma}\|_{\clH}\,\|\delta\bfx\|_{\clH}\Bigr).
\label{eq:deltabound}
\end{align}

Next, we use a duality argument to estimate the norm of $\delta \bfp$. Let~$\bfg\in\clH$ be arbitrary and let~$\bfy=\bfy(\bfg)$ solve~\eqref{eq:g}, then we obtain
\begin{align}
\|\delta \bfp\|_{\clH} &= \sup_{\|g\|_{\clH}\le1} \langle \delta \bfp, \bfg\rangle_{\clH}
=\sup_{\|g\|_{\clH}\le1} \langle \delta \bfp, \dot{\bfy} - \bfA_{\varSigma} \bfy + \bfB \bfB^\top \bfPi_{\varSigma} \bfy \rangle_{\clH}\notag\\
&= \sup_{\|g\|_{\clH}\le1}\big( \langle -\dot{\delta \bfp} - \bfA_{\varSigma}^\top \delta \bfp, \bfy \rangle_{\clH}+ \langle \bfB^\top \delta \bfp, \bfB^\top \bfPi_{\varSigma} \bfy \rangle_{\clH} \big)\notag\\
&= \sup_{\|g\|_{\clH}\le1}\big(\langle \delta\bfA^\top \frac1 N \clE p_{\sigma}, \bfy \rangle_{\clH} + \frac1 N \langle \delta \bfx, \bfy\rangle_{\clH} - \alpha \langle \delta u, \bfB^\top \bfPi_{\varSigma} \bfy \rangle_{\clH} \big),\notag
\end{align}
where we used~\eqref{eq:optcondiffC} and~\eqref{eq:optcondiffE}.
Hence, using~\eqref{eq:Pinorm} and~\eqref{eq:w-est}, we arrive at
\begin{align}
\begin{split}
\|\delta \bfp\|_{\clH} &\le M_s \Bigl( \frac1 N \|\delta\bfA\|\, \|\clE p_{\sigma}\|_{\clH} + \frac1 N \|\delta \bfx\|_{\clH} +\alpha\beta_2^2 \|\bfB\delta u\|_{\clH}\Bigr),\!\label{eq:deltap}
\end{split}
\end{align}
where we used~$\langle \delta u, \bfB^\top \bfPi_{\varSigma} \bfy \rangle_{\bbR^{nN}}\le
\|\bfB\delta u\|_{\bfPi_{\varSigma}}\|\bfy\|_{\bfPi_{\varSigma}}$. By~\eqref{eq:deltabound} and~\eqref{eq:deltap} we find
\begin{align}
\delta\clJ&\le \|\delta\bfA\| \, \|\clE x_{\sigma}\|_{\clH}M_s\Big(\frac1 N \|\delta\bfA\|\, \|\clE p_{\sigma}\|_{\clH} + \frac1 N \|\delta \bfx\|_{\clH} + \alpha\beta_2^2\, \|\bfB\|\, \|\delta u\|_{\clU}\Big)\notag\\
&\quad+ \frac1 N\|\delta\bfA\|\, \|\clE p_{\sigma}\|_{\clH}\,\|\delta\bfx\|_{\clH}\notag\\
&\le \|\delta\bfA\|^2 \, \frac1 N M_s\,\|\clE x_{\sigma}\|_{\clH} \,\|\clE p_{\sigma}\|_{\clH} 
+ \|\delta\bfA\|\,\frac1 N \|\clE p_{\sigma}\|_{\clH}\,\|\delta\bfx\|_{\clH}\notag\\
&\quad+ \|\delta\bfA\|\,M_s\,\|\clE x_{\sigma}\|_{\clH}\,\Big(\frac1 N\|\delta \bfx\|_{\clH}+ \alpha\beta_2^2\, \|\bfB\|\, \|\delta u\|_{\clU}\Big)\notag\\
&\le \|\delta\bfA\|^2 \frac1 N\,M_s\, \|\clE x_{\sigma}\|_{\clH} \,\|\clE p_{\sigma}\|_{\clH}+ \|\delta\bfA\|\sqrt{ N}\fkC_{M_s,x_{\sigma},p_{\sigma},\beta_2,\bfB}\Big(\frac1 N\|\delta \bfx\|_{\clH} + \alpha\|\delta u\|_{\clU}\Big),\notag
\end{align}
with $\fkC_{M_s,x_{\sigma},p_{\sigma},\beta_2,\bfB}$ as in~\eqref{eq:big-smallfkC0} (where we used~$\|\clE x_{\sigma}\| = \sqrt{N} \|x_{\sigma}\|$, $\|\clE p_{\sigma}\| = \sqrt{N} \|p_{\sigma}\|$).

Using~$\frac1 N\|\delta \bfx\|_{\clH} + \alpha\|\delta u\|_{\clU}\le\frac{1+\sqrt{\alpha N}}{\sqrt N}(\frac1{\sqrt N}\|\delta \bfx\|_{\clH} + \sqrt{\alpha}\|\delta u\|_{\clU})$  we further estimate
\begin{align}
\delta\clJ
&\le \|\delta\bfA\|^2 \frac1 N\,M_s\, \|\clE x_{\sigma}\|_{\clH} \, \|\clE p_{\sigma}\|_{\clH} + \|\delta\bfA\|^2  \fkC_{M_s,x_{\sigma},p_{\sigma},\beta_2,\bfB}^2 (1+\sqrt{\alpha N})^2
 \notag\\
&\quad+ \frac{1}{4}\bigg(\frac{1}{\sqrt{N}}\|\delta \bfx\|_{\clH} + \sqrt\alpha\|\delta u\|_{\clU}\bigg)^2\notag\\
&\le \|\delta\bfA\|^2\Bigl( \frac1 NM_s\|\clE x_{\sigma}\|_{\clH} \, \|\clE p_{\sigma}\|_{\clH} + (1+\sqrt{\alpha N})^2 \fkC_{M_s,x_{\sigma},p_{\sigma},\beta_2,\bfB}^2\Bigr)
+ \frac{1}{2}\delta\clJ,\notag
\end{align}
which leads to
\begin{align}
\frac{1}{2} \delta\clJ\le \|\delta\bfA\|^2 \bigg(M_s \| x_{\sigma}\|_{L^2( (0,+\infty);\bbR^{n})} \, \| p_{\sigma}\|_{L^2( (0,+\infty);\bbR^{n})} + (1+\sqrt{\alpha N})^2 \fkC_{M_s,x_{\sigma},p_{\sigma},\beta_2,\bfB}^2\bigg),\notag
\end{align}
which ends the proof. Note that~$\frac{1}{2}\delta\clJ=\clJ(\delta\bfx,\delta u)$, by definition of~$\delta\clJ$ in~\eqref{eq:deltabound}.\end{proof}

\begin{corollary}\label{coro:big-small1}
Let $(\bfx,u_{\varSigma})$ be the minimizer of~\eqref{eq:extobj} subject to~\eqref{eq:extsys} and let $(x_{\sigma},u_{\sigma})$ be the minimizer of~\eqref{eq:singlecost} subject to \eqref{eq:syssingle} with~$\bfx(0)=\clE x_{\sigma}(0)$. Then, there holds
\begin{align}\notag
0\le \clJ(\clE x_{\sigma}, u_{\sigma})-\clJ(\bfx,u_{\varSigma}) 
\le \|\bfA_{\varSigma} - \bfA_{\sigma}\| \fkC_{\bfx,x_{\sigma},u_{\varSigma},u_{\sigma}, N} \sqrt{\fkC_{M_s,x_{\sigma},p_{\sigma},\alpha,\beta_2,\bfB, N}},
\end{align}
where~$\fkC_{M_s,x_{\sigma},p_{\sigma},\alpha,\beta_2,\bfB, N}$ is as in~\eqref{eq:big-smallfkC1} and
\begin{equation}\label{eq:big-small1fkC}
\fkC_{\bfx,x_{\sigma},u_{\varSigma},u_{\sigma}, N} \coloneqq \frac{1}{\sqrt{ N}}\|\bfx + \clE x_{\sigma} \|_{L^2( (0,+\infty);\bbR^{nN})}+\sqrt{\alpha}\| u_{\varSigma} + u_{\sigma}\|_{L^2( (0,+\infty);\bbR^{m})}.
\end{equation}
\end{corollary}

\begin{proof} The inequality~$0\le \clJ(\clE x_{\sigma}, u_{\sigma})-\clJ(\bfx,u_{\varSigma})$ holds because~$(\bfx,u_{\varSigma})$ minimizes~$\clJ$. Next, with~$\clH\coloneqq L^2( (0,+\infty);\bbR^{nN})$ and~$\clU\coloneqq L^2( (0,+\infty);\bbR^{m})$, there holds
\begin{align}
&\clJ(\clE x_{\sigma}, u_{\sigma})-\clJ(\bfx,u_{\varSigma})\notag\\
 &\quad= \Bigl|\frac{1}{2 N} \langle \bfx - \clE x_{\sigma}, \bfx + \clE x_{\sigma} \rangle_{\clH} + \frac{\alpha}{2} \langle u_{\varSigma} - u_{\varSigma,\sigma} , u_{\varSigma} + u_{\varSigma,\sigma}\rangle_{\clU}\Bigr| \notag\\
&\quad\le \frac{1}{2 N} \| \bfx - \clE x_{\sigma}\|_{\clH} \|\bfx + \clE x_{\sigma} \|_{\clH}+ \frac{\alpha}{2} \| u_{\varSigma} - u_{\sigma} \|_{\clU}\| u_{\varSigma} + u_{\sigma}\|_{\clU}\notag\\
&\quad\le \Big(\frac{1}{\sqrt{ N}}\|\bfx + \clE x_{\sigma} \|_{\clH}+\sqrt{\alpha}\| u_{\varSigma} + u_{\sigma}\|_{\clU}\Big) \Big( \frac{1}{2\sqrt{ N}} \| \bfx - \clE x_{\sigma}\|_{\clH} +  \frac{\sqrt{\alpha}}{2} \| u_{\varSigma} - u_{\sigma} \|_{\clU}\Big)\notag\\ 
&\quad= \fkC_{\bfx,x_{\sigma},u_{\varSigma},u_{\sigma}, N} \sqrt{\clJ(\bfx - \clE x_{\sigma},u_{\varSigma}-u_{\sigma})}\le \fkC_{\bfx,x_{\sigma},u_{\varSigma},u_{\sigma}, N} \|\bfA_{\varSigma} - \bfA_{\sigma}\| \sqrt{C},\notag
\end{align}
with~$\fkC_{\bfx,x_{\sigma},u_{\varSigma},u_{\sigma}, N}$ as in~\eqref{eq:big-small1fkC} and~$C=\fkC_{M_s,x_{\sigma},p_{\sigma},\alpha,\beta_2,\bfB, N}$ as in~\eqref{eq:big-smallfkC1}. Here we have used the inequalities $ab+cd \le (a+c) (b+d)$ and~$\frac{\sqrt{a}}{2} + \frac{\sqrt{b}}{2} \le \sqrt{\frac{a+b}{2}}$, holding for real numbers $a,b,c,d\geq 0$.\end{proof}

\begin{theorem}\label{thm:big-small2}
Let $(\bfx,u_{\varSigma})$ be the minimizer of functional \eqref{eq:extobj} subject to~\eqref{eq:extsys} and let $(x_{\sigma},u_{\sigma})$ be the minimizer of \eqref{eq:singlecost} subject to \eqref{eq:syssingle}. Then, for each $\vartheta> \|\bfA_{\varSigma} - \bfA_{\sigma}\|$,
\begin{align}\notag
& 0\le\clJ(\clE x_{\sigma},u_{\sigma})-\clJ(\bfx,u_{\varSigma}) \le \|\delta\bfA \| \frac{2}{\vartheta-\|\delta\bfA \|}\bigg(\vartheta^2\fkC_{M_s,x_{\sigma},p_{\sigma},\alpha,\beta_2,\bfB, N} + \clJ(\bfx,u_{\varSigma}) \bigg),
\end{align}
where~$\|\delta\bfA \|\coloneqq\|\bfA_{\varSigma} - \bfA_{\sigma}\|$ and~$\fkC_{M_s,x_{\sigma},p_{\sigma},\alpha,\beta_2,\bfB, N}$ is as in~\eqref{eq:big-smallfkC1}.
\end{theorem}

\begin{proof}
Let~$\clH\coloneqq L^2( (0,+\infty);\bbR^{nN})$ and~$\clU\coloneqq L^2( (0,+\infty);\bbR^{m})$, and recall from Corollary~\ref{coro:big-small1}, with~$\fkC_{\bfx,x_{\sigma},u_{\varSigma},u_{\sigma}, N}$ as in~\eqref{eq:big-small1fkC} and~$C\coloneqq\fkC_{M_s,x_{\sigma},p_{\sigma},\alpha,\beta_2,\bfB, N}$ as in~\eqref{eq:big-smallfkC1}, that
\begin{align}\label{JbfxJx-1}
0\le \clJ(\clE x_{\sigma},u_{\sigma})-\clJ(\bfx,u_{\varSigma}) \le \fkC_{\bfx,x_{\sigma},u_{\varSigma},u_{\sigma}, N} \|\bfA_{\varSigma} - \bfA_{\sigma}\| \sqrt{C}.
\end{align}
Further, by triangle and Young inequalities,
\begin{align}
&\fkC_{\bfx,x_{\sigma},u_{\varSigma},u_{\sigma}, N}\sqrt{C}
\le\Bigl( \frac{1}{ \sqrt N}\|\bfx\|_{\clH} + \sqrt\alpha\|u_{\varSigma}\|_{\clU} + \frac{1}{ \sqrt N}\|\clE x_{\sigma} \|_{\clH}+ \sqrt\alpha\|u_{\sigma}\|_{\clU}\Bigr)\sqrt{C}
\notag\\
&\quad\le\Bigl( \sqrt{4\clJ(\bfx,u_{\varSigma})} + \sqrt{4\clJ(\clE x_{\sigma},u_{\sigma})}\Bigr)\sqrt{C}\le\frac\varepsilon2\Bigl( \sqrt{\clJ(\bfx,u_{\varSigma})} + \sqrt{\clJ(\clE x_{\sigma},u_{\sigma})}\Bigr)^2+ \frac{2}\varepsilon C\notag\\
&\quad\le\varepsilon\clJ(\bfx,u_{\varSigma}) +\varepsilon \clJ(\clE x_{\sigma},u_{\sigma})+ \frac{2C}\varepsilon,\label{JbfxJx-2}
\end{align}
for an arbitrary~$\varepsilon>0$, where we used~$a+b=\sqrt{(a+b)^2}\le\sqrt{2(a^2+b^2)}$, for nonnegative~$a,b$ in the second step. By combining~\eqref{JbfxJx-1} and~\eqref{JbfxJx-2}, we conclude that
\begin{align}
& (1-\varepsilon\|\bfA_{\varSigma} - \bfA_{\sigma}\|)\Bigl( \clJ(\clE x_{\sigma},u_{\sigma})-\clJ(\bfx,u_{\varSigma}) \Bigr)\le \|\bfA_{\varSigma} - \bfA_{\sigma}\| \bigg(2 \varepsilon \clJ(\bfx,u_{\varSigma}) +\frac{ 2C}{\varepsilon}\bigg),\notag
\end{align}
and, we end the proof by taking $\varepsilon =\vartheta^{-1}$ and using~$\frac{2}{1-\vartheta^{-1}\|\delta\bfA\|}= \frac{2\vartheta}{\vartheta-\|\delta\bfA\|}$.
\end{proof}

\subsection{The Cost of the Proposed Feedback Control} 
The following result compares the minimal cost associated with the solution of the extended system to the cost  associated with the solution resulting from the feedback control $u_{\varSigma,\sigma}$ that we propose in~\eqref{eq:fbcontrol}. It provides a bound for the difference of these costs, i.e., it quantifies the suboptimality of $(u_{\varSigma,\sigma},x_{\varSigma,\sigma})$,
under the assumption that the solution $x_{\varSigma,\sigma}$ of~\eqref{eq:avclsys} is in~$W( (0,+\infty),\bbR^{n})$. Later on, in Corollary~\ref{coro:stableFeedback}, we provide conditions ensuring that $x_{\varSigma,\sigma} \in W( (0,+\infty),\bbR^{n})$.

\begin{lemma}\label{lem:big-rob}
Let $(u_{\varSigma},\bfx)$ be the minimizer of the extended system \eqref{eq:extobj} subject to \eqref{eq:extsys}, and for a given parameter $\sigma$ let $(u_{\varSigma,\sigma},x_{\varSigma,\sigma})$ be given by \eqref{eq:fbcontrol} and \eqref{eq:avclsys}, with~$\bfx(0)=\clE x_{\varSigma,\sigma}(0)$. Further, assume that $x_{\varSigma,\sigma} \in W( (0,+\infty),\bbR^{n})$. Then, there holds
\begin{align}
0 \le \clJ(\clE x_{\varSigma,\sigma},u_{\varSigma,\sigma}) - \clJ(\bfx,u_{\varSigma})
\le \beta_2^2\,\|\bfA_{\sigma} - \bfA_{\varSigma}\|\,\|\clE x_{\varSigma,\sigma}\|^2_{L^2( (0,+\infty),\bbR^{nN})}.\notag
\end{align}
\end{lemma}

\begin{proof}
Let~$\clH\coloneqq L^2( (0,+\infty);\bbR^{nN})$, and let~$x_{\circ}\coloneqq x_{\varSigma,\sigma}(0)$ and~$\bfx_\circ\coloneqq\clE x_{\circ}$. We know that the minimal value of~$\clJ(\bfx,u_{\varSigma})$ subject to the extended system \eqref{eq:extsys} is equal to $\frac12 \langle \bfx_{\circ}, \bfPi_{\varSigma} \bfx_{\circ}\rangle$. Thus, using~$x_{\varSigma,\sigma} \in W( (0,+\infty),\bbR^{n})$ we find that~$\lim_{t\to+\infty}x_{\varSigma,\sigma}(t)=0$ (cf.~\cite[Proof of~(2.8), Thm.~2.4]{casas2022infinite_2}) and
\begin{align}
\clJ(\bfx,u_{\varSigma}) &= \frac12 \langle \clE x_{\circ}, \bfPi_{\varSigma} \clE x_{\circ}\rangle= \frac12 \int_0^{\infty} -\frac{\rmd}{\rmd t}\langle \clE x_{\varSigma,\sigma}(t), \bfPi_{\varSigma} \clE x_{\varSigma,\sigma}(t)\rangle {\rm dt}\notag\\
&= -\frac12 \langle\clE \dot{x}_{\varSigma,\sigma}, \bfPi_{\varSigma} \clE x_{\varSigma,\sigma}\rangle_\clH -\frac12 \langle \clE x_{\varSigma,\sigma}, \bfPi_{\varSigma} \clE \dot{x}_{\varSigma,\sigma} \rangle_\clH.\notag
\end{align}
Using~\eqref{eq:avclsys} we find that
\begin{align}
&\clJ(\bfx,u_{\varSigma})\notag \\&= -\frac12 \langle(\clE^\top\bfPi_{\varSigma}\clE (\clA_\sigma-\frac{1}{\alpha}B \bfB^\top \bfPi_{\varSigma} \clE)+(\clA_\sigma-\frac{1}{\alpha}B \bfB^\top \bfPi_{\varSigma} \clE)^\top\clE^\top\bfPi_{\varSigma}\clE) x_{\varSigma,\sigma},  x_{\varSigma,\sigma}\rangle_\clH\notag\\
&= -\frac12 \langle\bfZ  x_{\varSigma,\sigma},   x_{\varSigma,\sigma}\rangle_\clH\notag
 \end{align}
 with~$\bfZ\coloneqq \clE^\top\bfPi_{\varSigma} \clE\clA_\sigma+ (\clE\clA_\sigma)^\top\bfPi_{\varSigma}\clE-\frac{1}{\alpha}\clE^\top\bfPi_{\varSigma} \clE B \bfB^\top \bfPi_{\varSigma} \clE- \frac{1}{\alpha}\clE^\top\bfPi_{\varSigma}  \bfB(\clE B)^\top\bfPi_{\varSigma}\clE.$
 
 Using the identities~$\clE \clA_\sigma=\bfA_{\sigma} \clE$ and~$\clE B=\bfB$, we find
 \begin{align}
 \bfZ &= \clE^\top(\bfPi_{\varSigma} \bfA_{\sigma}+ \bfA_{\sigma}^{\top} \bfPi_{\varSigma}-\frac{2}{\alpha}\bfPi_{\varSigma} \bfB \bfB^\top \bfPi_{\varSigma})\clE\notag\\
 &= \clE^\top\Bigl(\bfPi_{\varSigma} (\bfA_{\sigma}-\bfA_{\varSigma})+ (\bfA_{\sigma}^{\top}-\bfA_{\varSigma}^\top) \bfPi_{\varSigma}-\frac{1}{\alpha}\bfPi_{\varSigma} \bfB \bfB^\top \bfPi_{\varSigma} -\frac1{N}\Id_{nN} \Bigr)\clE,\notag
 \end{align}
where we used the Riccati equation \eqref{eq:bigRiccati}. This leads us to
 \begin{align}
 \clJ(\bfx,u_{\varSigma})&= \frac{1}{2\alpha} \langle \bfPi_{\varSigma} \bfB \bfB^\top \bfPi_{\varSigma} \clE x_{\varSigma,\sigma}, \clE x_{\varSigma,\sigma} \rangle_\clH + \frac{1}{2 N} \langle \clE x_{\varSigma,\sigma}, \clE x_{\varSigma,\sigma} \rangle_\clH\notag\\
 &\quad-\frac{1}{2}\langle\bfD\clE  x_{\varSigma,\sigma}, \clE x_{\varSigma,\sigma}\rangle_\clH\notag
\end{align} 
with
$\bfD\coloneqq \bfPi_{\varSigma}(\bfA_{\sigma} - \bfA_{\varSigma})+(\bfA_{\sigma} - \bfA_{\varSigma})^\top\bfPi_{\varSigma}.$
Finally, recalling~\eqref{eq:fbcontrol}, we have 
\begin{align}
\clJ(\clE x_{\varSigma,\sigma},u_{\varSigma,\sigma})  -  \clJ(\bfx,u_{\varSigma})
& =\frac{1}{2}\langle\bfD\clE  x_{\varSigma,\sigma}, \clE x_{\varSigma,\sigma}\rangle_\clH\le \beta_2^2\,\|\bfA_{\sigma} - \bfA_{\varSigma}\| \|\clE x_{\varSigma,\sigma}\|_\clH^2,\notag
\end{align}
where we used~\eqref{eq:Pinorm}.
\end{proof}

\begin{theorem}\label{thm:final}
For a given parameter $\sigma\in\bbR$ let $(x_{\sigma},u_{\sigma})$ be the minimizer of \eqref{eq:singlecost} subject to \eqref{eq:syssingle}, and let $(x_{\varSigma,\sigma},u_{\varSigma,\sigma})$ be given by \eqref{eq:fbcontrol} and \eqref{eq:avclsys}. If we have that~$x_{\varSigma,\sigma} \in W( (0,+\infty);\bbR^{n})$, then there holds the estimate
\begin{align}
 |\clJ(\clE x_\sigma, u_\sigma) - \clJ(\clE x_{\varSigma,\sigma}, u_{\varSigma,\sigma})|
 &\le \fkC_{\bfx,x_{\sigma},u_{\varSigma},u_{\sigma},M_s,p_{\sigma},\alpha,\beta_2,\bfB, N} \|\bfA_{\sigma} -\bfA_{\varSigma}\|,\notag
\end{align}
with
\begin{equation}\label{eq:finalfkC}
\fkC_{\bfx,x_{\sigma},u_{\varSigma},u_{\sigma},M_s,p_{\sigma},\alpha,\beta_2,\bfB, N} \coloneqq \max(C_2 \sqrt{C_1},\\\beta_2^2 \|\clE x_{\varSigma,\sigma}\|^2_{L^2( (0,+\infty),\bbR^{nN})}),
\end{equation}  where $C_2\coloneqq \fkC_{\bfx,x_{\sigma},u_{\varSigma},u_{\sigma}, N}$ and $C_1\coloneqq \fkC_{M_s,x_{\sigma},p_{\sigma},\alpha,\beta_2,\bfB, N}$ are the constants defined in \eqref{eq:big-small1fkC} and \eqref{eq:big-smallfkC1}, respectively.
\end{theorem}

\begin{proof}
With~$\clH\coloneqq L^2( (0,+\infty),\bbR^{nN})$ and~$\delta\clJ\coloneqq\clJ(\clE x_\sigma, u_\sigma) - \clJ(\clE x_{\varSigma,\sigma}, u_{\varSigma,\sigma})$, we have
\begin{align}
 |\delta\clJ|&\le |\clJ(\clE x_\sigma, u_\sigma) - \clJ(\bfx,u_{\varSigma})| + |\clJ(\bfx,u_{\varSigma}) - \clJ(\clE x_{\varSigma,\sigma}, u_{\varSigma,\sigma})|\label{deltaJ-thmmain}\\
  &\le\|\bfA_{\sigma} -\bfA_{\varSigma}\| \fkC_{\bfx,x_{\sigma},u_{\varSigma},u_{\sigma}, N} \sqrt{\fkC_{M_s,x_{\sigma},p_{\sigma},\alpha,\beta_2,\bfB, N}} 
  +  \beta_2^2 \|\bfA_{\sigma}-\bfA_{\varSigma}\| \|\clE x_{\varSigma,\sigma}\|^2_{\clH}\notag\\
 &\le \fkC_{\bfx,x_{\sigma},u_{\varSigma},u_{\sigma},M_s,p_{\sigma},\alpha,\beta_2,\bfB, N} \|\bfA_{\sigma} -\bfA_{\varSigma}\|,\notag
\end{align}
where we used Corollary~\ref{coro:big-small1}  and Lemma~\ref{lem:big-rob}. 
\end{proof}

\begin{corollary}\label{cor:final}
Given a parameter $\sigma\in\bbR$, let $(x_{\sigma},u_{\sigma})$ be the minimizer of \eqref{eq:singlecost} subject to \eqref{eq:syssingle}, and let $(x_{\varSigma,\sigma},u_{\varSigma,\sigma})$ be given by \eqref{eq:fbcontrol} and \eqref{eq:avclsys}. If~$x_{\varSigma,\sigma} \in W( (0,+\infty);\bbR^{n})$, then it follows that: for each~$\vartheta>\|\bfA_{\varSigma} - \bfA_{\sigma}\|$ there holds the estimate
\begin{align}
 &|\clJ(\clE x_\sigma, u_\sigma) - \clJ(\clE x_{\varSigma,\sigma}, u_{\varSigma,\sigma})|\notag\\
 &\le \|\delta\bfA \| \biggl(  \frac{2}{\vartheta-\|\delta\bfA \|}\Bigl(\vartheta^2\fkC_{M_s,x_{\sigma},p_{\sigma},\alpha,\beta_2,\bfB, N} + \clJ(\bfx,u_\varSigma) \Bigr)+ \beta_2^2 \|\clE x_{\varSigma,\sigma}\|^2_{L^2( (0,+\infty),\bbR^{nN})}\biggr),\notag
\end{align}
where~$\|\delta\bfA \|\coloneqq\|\bfA_{\varSigma} - \bfA_{\sigma}\|$ and~$\fkC_{M_s,x_{\sigma},p_{\sigma},\alpha,\beta_2,\bfB, N}$ is defined in~\eqref{eq:big-smallfkC1}.
\end{corollary}
\begin{proof}
We follow the steps in the proof of Theorem~\ref{thm:final} and write~\eqref{deltaJ-thmmain}, from which we can obtain
the desired estimate, by using Theorem~\ref{thm:big-small2} and Lemma~\ref{lem:big-rob}. 
\end{proof}

\subsection{Stabilization Property of the Proposed Feedback Control} 
In Lemma~\ref{lem:big-rob}, Theorem~\ref{thm:final} and Corollary~\ref{cor:final} we assumed that $x_{\varSigma,\sigma} \in W( (0,+\infty),\bbR^{n})$. Here, we provide sufficient conditions which ensure that this assumption holds.

\begin{lemma}\label{lem:expstab}
Given a parameter $\sigma\in\bbR$, which is not necessarily an element of $\varSigma$, for the pair $(x_{\varSigma,\sigma},u_{\varSigma,\sigma})$ given by \eqref{eq:fbcontrol} and \eqref{eq:avclsys}, there holds that, for $t\ge 0$,
 \begin{align}
  \frac{\mathrm d}{\mathrm dt} \|\clE x_{\varSigma,\sigma}\|^2_{\bfPi_{\varSigma}} \leq \left( -\frac{1}{ N\beta_2^2} + \frac{2\beta_2}{\beta_1}\|\bfA_{\sigma} - \bfA_{\varSigma}\|\right) \|\clE x_{\varSigma,\sigma}\|_{\bfPi_{\varSigma}}^2- \frac{1}{\alpha}\|\bfB^\top \bfPi_{\varSigma}\clE x_{\varSigma,\sigma}\|^2.\notag
 \end{align}
\end{lemma}
\begin{proof}
For simplicity, let us denote~$z\coloneqq x_{\varSigma,\sigma}$. By direct computations we find
\begin{align}
  \frac{\mathrm d}{\mathrm dt} \|\clE z\|^2_{\bfPi_{\varSigma}} &=\frac{\mathrm d}{\mathrm dt} \langle \clE z, \bfPi_{\varSigma} \clE z\rangle = \langle \clE \dot z, \bfPi_{\varSigma} \clE z\rangle + \langle \clE z, \bfPi_{\varSigma} \clE \dot z \rangle\notag \\
&=\langle \clE \clA_{\sigma} z  - \frac{1}{\alpha} \bfB \bfB^\top \bfPi_{\varSigma} \clE z , \bfPi_{\varSigma} \clE z \rangle  + \langle \clE z , \bfPi_{\varSigma} \clE \clA_{\sigma} z  - \frac{1}{\alpha} \bfPi_{\varSigma} \bfB \bfB^\top \bfPi_{\varSigma} \clE z  \rangle\notag\\
 &=\langle \bfA_{\varSigma} \clE z  - \frac{1}{\alpha} \bfB \bfB^\top \bfPi_{\varSigma} \clE z ,  \bfPi_{\varSigma} \clE z \rangle + \langle (\clE \clA_{\sigma} - \bfA_{\varSigma} \clE) z , \bfPi_{\varSigma} \clE z  \rangle\notag \\
&\quad+ \langle \clE z , \bfPi_{\varSigma}(\clE \clA_{\sigma} - \bfA_{\varSigma}\clE)z  \rangle+ \langle \clE z , \bfPi_{\varSigma} \bfA_{\varSigma} \clE z  -\frac{1}{\alpha} \bfPi_{\varSigma} \bfB \bfB^\top \bfPi_{\varSigma} \clE z  \rangle. \notag
  \end{align}
Thus, with~$\delta\bfA\coloneqq  \bfA_{\sigma}- \bfA_{\varSigma}$, we find~$\delta\bfA\clE=\clE \clA_{\sigma} - \bfA_{\varSigma} \clE$ and
 \begin{align}
  \frac{\mathrm d}{\mathrm dt} \|\clE z\|^2_{\bfPi_{\varSigma}} &=\langle \bfPi_{\varSigma} \bfA_{\varSigma} \clE z  - \frac{1}{\alpha} \bfPi_{\varSigma} \bfB \bfB^\top \bfPi_{\varSigma} \clE z , \clE z  \rangle+ \langle \delta\bfA\clE z , \bfPi_{\varSigma}\clE z  \rangle\\
 &\quad+ \langle \clE z , \bfPi_{\varSigma} \delta\bfA\clE z  \rangle + \langle \bfA_{\varSigma}^\top \bfPi_{\varSigma} \clE z  - \frac{1}{\alpha} \bfPi_{\varSigma} \bfB \bfB^\top \bfPi_{\varSigma} \clE z , \clE z  \rangle\notag \\
 &= - \frac{1}{\alpha} \langle \bfPi_{\varSigma} \bfB \bfB^\top \bfPi_{\varSigma} \clE z , \clE z \rangle - \frac{1}{ N} \langle \clE z , \clE z  \rangle+ \langle\delta\bfA\clE  z , \bfPi_{\varSigma}\clE z  \rangle + \langle \clE z , \bfPi_{\varSigma} \delta\bfA\clE z  \rangle\\
 &= -\frac{1}{\alpha}\|\bfB^\top \bfPi_{\varSigma} \clE z \|^2 - \frac{1}{ N} \|\clE z \|^2 + 2 \langle \delta\bfA\clE  z , \bfPi_{\varSigma}\clE z  \rangle,\notag 
\end{align}
where we used the Riccati equation \eqref{eq:bigRiccati}. Now, from 
 \begin{align}
2 \langle \delta\bfA\clE  z , \bfPi_{\varSigma}\clE z  \rangle=2\langle \delta\bfA\clE  z , \clE z  \rangle_{\bfPi_{\varSigma}} \le2\|\delta\bfA\clE  z\|_{\bfPi_{\varSigma}}\|\clE z\|_{\bfPi_{\varSigma}} \le2\beta_2\|\delta\bfA\|\|\clE  z\|\|\clE  z\|_{\bfPi_{\varSigma}},\notag
 \end{align}
the assertion follows from~\eqref{eq:Pinorm}.
 \end{proof}
 
As a consequence, the following result on the stability of~\eqref{eq:avclsys} follows.

\begin{corollary}\label{coro:stableFeedback}
For a given parameter $\sigma \in \bbR$, let
\begin{align}\label{eq:smallness}
 \|\bfA_{\sigma} - \bfA_{\varSigma}\| < \frac{\beta_1}{2 N\beta_2^3}.
\end{align}
Then, with the feedback~\eqref{eq:linFB}, the system \eqref{eq:avclsys} is stable. More precisely, for~$t\ge0$, 
\begin{subequations}\label{eq:expstab}
\begin{align}\label{eq:expstab1}
 \|\clE x_{\varSigma,\sigma}(t)\|_{\bfPi_{\varSigma}}^2 &\le \rme^{-\lambda t}\|\clE x_{\varSigma,\sigma}(0)\|_{\bfPi_{\varSigma}}^2,\\
 \|x_{\varSigma,\sigma}(t)\|^2  &\le \Big(\frac{\beta_2}{\beta_1}\Big)^2 \rme^{-\lambda t} \|x_{\varSigma,\sigma}\|^2, \label{eq:expstab2}\\
\mbox{with}\quad\lambda &\coloneqq \left(\frac{1}{N\beta_2^2} - \frac{2\beta_2}{\beta_1} \|\bfA_{\sigma} - \bfA_{\varSigma}\|\right) > 0. \label{eq:expstab-lam}
\end{align}
\end{subequations}
\end{corollary}
\begin{proof}
From Lemma~\ref{lem:expstab} and \eqref{eq:smallness} we have that
 \begin{align}\label{eq:expstab2a}
   \frac{\mathrm d}{\mathrm dt} \|\clE x_{\varSigma,\sigma}\|^2_{\bfPi_{\varSigma}} \leq -\lambda \|\clE x_{\varSigma,\sigma}\|_{\bfPi_{\varSigma}}^2 - \frac{1}{\alpha} \|\bfB^\top \bfPi_{\varSigma} \clE x_{\varSigma,\sigma}\|^2,\quad\mbox{for}\quad t>0,
 \end{align}
 with $\lambda$ as in~\eqref{eq:expstab-lam}, which implies~\eqref{eq:expstab1}. Finally, we use~\eqref{eq:Pinorm} to obtain~\eqref{eq:expstab2}. \end{proof}

\subsection{The Riccati Feedback for the Average of Parameters}
Above in Corollary~\ref{coro:stableFeedback} we give the sufficient condition~\eqref{eq:smallness} for the stabilizing properties of the feedback operator~$-\frac{1}{\alpha}\bfB^\top\bfPi_\varSigma\clE$, where~$\bfPi$ solves~\eqref{eq:bigRiccati}. The condition requires, in particular the smallness of~$\|\delta\bfA\|=\|\bfA_{\sigma} - \bfA_{\varSigma}\|$. Here, firstly, we recall that for small~$\|\delta\bfA\|$ we also have the stabilizing property for the average feedback~$-\frac{1}{\alpha}B^\top\Pi_{\bar\sigma}$, with~$\bar\sigma\coloneqq\frac1N\sum_{i=1}^N\sigma_i$, where~$\Pi_{\bar\sigma}$ solves~\eqref{eq:smallRiccati}. This latter feedback is easier to compute since the size of~$-\frac{1}{\alpha}B^\top\Pi_{\bar\sigma}$ is smaller than the size of~$-\frac{1}{\alpha}\bfB^\top\bfPi_\varSigma\clE$, $n<nN$. This raises the question: why using~$-\frac{1}{\alpha}\bfB^\top\bfPi_\varSigma\clE$ instead of~$-\frac{1}{\alpha}B^\top\Pi_{\bar\sigma}$. The answer is given in the numerical simulations hereafter, showing that the former feedback is more robust. The same simulations also indicate that, the condition~\eqref{eq:smallness} in Corollary~\ref{coro:stableFeedback} may be not necessary, see~Remark \ref{rem:nnecessary} below.

\begin{theorem}\label{thm:stableFeed_avSigma}
For a given parameter $\sigma \in \bbR$, let
\begin{align}\label{eq:smallness_new}
 \|\clA_{\sigma} - \clA_{\bar\sigma}\| < (2\|\Pi_{\bar\sigma}\|)^{-1},
\end{align}
where~$\bar\sigma\coloneqq\frac1N\sum_{i=1}^N\sigma_i$ and~$\Pi_{\bar\sigma}\succ0$ solves~\eqref{eq:smallRiccati} (with~$\clA_{\bar\sigma}$ in place of~$\clA_{\sigma}$). Then,
\begin{equation}\label{sys:avsigsig}
\dot x=\Bigl(\clA_\sigma-\frac{1}{\alpha}BB^\top\Pi_{\bar\sigma}\Bigr)x,\qquad x(0)=x_\circ,
\end{equation}
is a stable system. More precisely, there exists a constant~$C\ge1$ such that 
\begin{align}
 \|x(t)\|^2 &\le C\rme^{-\lambda t}\|x(0)\|^2,\quad\mbox{for all~$t\ge0$},\label{eq:expstab-lam-avsig}
\end{align}
with~$\lambda=1-2\|\Pi_{\bar\sigma}\|\|\clA_\sigma-\clA_{\bar\sigma}\|$.
\end{theorem}
\begin{proof}
Direct computations, using~\eqref{eq:smallRiccati},  lead us to
\begin{align}
\frac{\rmd}{\rmd t}\langle \Pi_{\bar\sigma} x,x\rangle&=-\frac{1}{\alpha}\|B^\top\Pi_{\bar\sigma}x\|^2-\|x\|^2 +\langle \Pi_{\bar\sigma}(\clA_\sigma-\clA_{\bar\sigma}) x,x\rangle+\langle \Pi_{\bar\sigma} x,(\clA_\sigma-\clA_{\bar\sigma})x\rangle\notag\\
&\le -\frac{1}{\alpha}\|B^\top\Pi_{\bar\sigma}x\|^2-(1-\|\Pi_{\bar\sigma}(\clA_\sigma-\clA_{\bar\sigma})+(\clA_\sigma-\clA_{\bar\sigma})^\top\Pi_{\bar\sigma}\|)\|x\|^2 \notag\\
&\le -\frac{1}{\alpha}\|B^\top\Pi_{\bar\sigma}x\|^2-(1-2\|\Pi_{\bar\sigma}\|\|\clA_\sigma-\clA_{\bar\sigma}\|)\|x\|^2 \le-\lambda\|x\|^2,\notag
\end{align}
which implies~\eqref{eq:expstab-lam-avsig}.
\end{proof}

\section{Examples: Ensemble Stabilizable Systems}\label{sec:enstabsys}
We collect examples of ensemble stabilizable systems, showing the stabilizability performance of the feedback~\eqref{eq:linFB}, thus illustrating the result in  Corollary~\ref{coro:stableFeedback}.

\subsection{The Oscillator}
Let us consider the differential equation
 \begin{align}\label{sys-osc0u}
\ddot \theta &=-\theta-{\sigma}\dot \theta + u,\qquad  \theta(0)= \theta_\circ,\qquad \dot\theta(0)= \theta_{\circ1}.
\end{align}
We assume that the restoring force of the spring is proportional to the position~$\theta(t)$, at time~$t>0$ with factor~$-1$, and that the friction, or damping, is proportional to the velocity~$\dot \theta(t)$ at time~$t>0$, where we are uncertain about the precise damping factor~$-\sigma$. Thus, we consider a finite ensemble~$\varSigma\coloneqq(\sigma_i)_{i=1}^N$ of possible values of~$\sigma$, and rewrite the second order equation in~\eqref{sys-osc0u}, for each~$\sigma_i$, as
\begin{align}\label{sys-oscu}
\dot x_{\sigma_i}&= \clA_{\sigma_i}x_{\sigma_i}+ Bu,\qquad x_{\sigma_i}(0)=x_\circ,\qquad1\le i \le N,
\end{align}
with operators~$\clA_{\sigma_i} = \begin{bmatrix}0&1\\-1&-\sigma_i\end{bmatrix}$, $B = \begin{bmatrix}0 \\1\end{bmatrix}$, initial condition~$x_\circ\coloneqq\begin{bmatrix}\theta_\circ\\\theta_{\circ1}\end{bmatrix} \in \bbR^2$, and corresponding states~$x_{\sigma_i}\coloneqq\begin{bmatrix}\theta &\dot\theta\end{bmatrix}^\top=\begin{bmatrix}\theta_{\sigma_i}& \dot\theta_{\sigma_i}\end{bmatrix}^\top$. 

Looking at~$\varSigma$ as a set of training parameters, we assume that it is chosen with pairwise distinct elements.

\begin{lemma}\label{lem:enscontr}
 Let the elements in $\varSigma$ be pairwise distinct (i.e., $\sigma_i \neq \sigma_j$ if $i \neq j$). Then, the ensemble of systems $(\clA_{\sigma_i},B)_{i=1}^N$ given in \eqref{sys-oscu} is ensemble controllable.
\end{lemma}

\begin{proof}
We use the Hautus test for controllability~\eqref{eq:HautusContr}. Firstly, we note that, for any given row vector~$v=[ v_1\;v_2\;\dots \;v_{2N}]\in\bbC^{1\times 2N}$, we have
\begin{align}
0=v\begin{bmatrix}\bfA_{\varSigma}-\lambda\Id_{nN} &\bfB\end{bmatrix}\quad\!\!\!\Longleftrightarrow\quad\!\!\!
\begin{cases}
-\lambda v_{2n-1}-v_{2n}=0,&\;\mbox{for all } 1\le n\le N;\\
v_{2n-1}+(\sigma_n-\lambda)v_{2n}=0,&\;\mbox{for all }1\le n\le N;\\
\sum_{n=1}^N v_{2n}=0.
\end{cases}\label{Hautus2}
\end{align}
Then, considering the right-hand-side of~\eqref{Hautus2}, we distinguish the following cases:
\begin{itemize}
\item[-] If~$\lambda=0$, from the first line we find~$v_{2n}=0$, for all~$1\le n\le N$ and, subsequently from the second line, ~$v_{2n-1}=0$, for all~$1\le n\le N$. Hence $v=0$.
\item[-] If~$\lambda\ne 0$, from the first and third lines it follows that~$\sum_{n=1}^N v_{2n-1}=0$. From the third line we have $\sum_{n=1}^N v_{2n} =0$, and subsequently from the second line, we have~$\sum_{n=1}^N \sigma_nv_{2n}=0$. Again, from the first line~$\sum_{n=1}^N \sigma_nv_{2n-1}=0$ and, subsequently, from the second line,~$\sum_{n=1}^N \sigma_n^2v_{2n}=0$. Repeating the argument we find that, with~$w\in\{w^{\bfe},w^{\bfo}\}$ where~$w^{\bfe}\in\bbR^{1\times N}$, and~$w^{\bfo}\in\bbR^{1\times N}$ are defined by
$w^{\bfe}_{n}=v_{2n}$ and $w^{\bfo}_{n}=v_{2n-1}$, it holds that~$w\clV_\varSigma=0$, where
\begin{align}\label{VanderSigma}
\clV_\varSigma\coloneqq\begin{bmatrix}
1&\sigma_1&\sigma_1^2&\dots&\sigma_1^{N-1}\\
1&\sigma_2&\sigma_2^2&\dots&\sigma_2^{N-1}\\
\vdots&\vdots&\dots&\vdots\\
1&\sigma_N&\sigma_N^2&\dots&\sigma_N^{N-1}
\end{bmatrix}.
\end{align}
Note that~$\clV_\varSigma$ is a Vandermonde matrix, which is nonsingular since $\sigma_i\ne\sigma_j$ for all $1\le i<j\le N$. In this case we conclude that~$w=0$, and thus~$v=0$.
\end{itemize}

Therefore, we conclude that~$\rank\begin{bmatrix}\bfA_{\varSigma}-\lambda\Id_{nN}&\bfB\end{bmatrix}=2N$, which ends the proof, due to the Hautus test~\eqref{eq:HautusContr}.
\end{proof}

Finally, note that, for a single parameter,  we also have that~$(\clA_\sigma,B)$ is controllable, due to the Kalman condition and 
$\rank \begin{bmatrix}
  B & A_{\sigma} B 
 \end{bmatrix} =\rank \begin{bmatrix}
  0&1\\1&- \sigma
 \end{bmatrix} =2$.


 \subsection{Multi-compartment Models}\label{sS:MComp}
We consider  a $3$-compartment model dynamics
 \begin{align}
\dot x&=\clA x+Bu,\qquad x(0)=x_\circ\in\bbR^{3\times1},\notag
\end{align}
with~$B\in\bbR^{3\times1}$ and~$\clA\in\bbR^{3\times3}$ in the form
\begin{align}\label{sys-comp-gen}
\clA \coloneqq \begin{bmatrix} -(b_{12}+b_{13}+a_1) & b_{21} &  b_{31} \\ b_{12} & -(b_{21} + b_{23}+ a_2) & b_{32} \\ b_{13}  & b_{23} & -(b_{31} + b_{32} + a_3)\end{bmatrix}
\end{align}
where $b_{ij}\ge0$ for $1\le i,j\le 3$ (see, e.g.,~\cite{farkasHWS21}).  If $a_i = 0$ for $1\le i\le 3$ the model is called closed, otherwise it is called open. We can interpret~$a_i>0$ as a sink in the~$i$-th compartment, and~$a_i<0$ as a source in the same compartment. The terms~$b_{ij}$ and~$b_{ji}$ indicate the ``flow'' between the~$i$-th and~$j$-th compartments.


\subsubsection{A Catenary Model with Closed Uncertain Dynamics}\label{sS:MComp-cat-closed}
 Consider the closed catenary system with uncertain flow between the second and third compartment
  \begin{equation}\notag
\clA_{\sigma_i}= \begin{bmatrix}-b_{12}&b_{21}&0\\b_{12}&-(b_{21}+b_{23})&b_{32}+\sigma_i\\0&b_{23}&-b_{32}-\sigma_i\end{bmatrix}, \qquad B = \begin{bmatrix} 1\\ 0 \\ 0 \end{bmatrix},
\end{equation}
with ensemble of parameters~$\varSigma=(\sigma_i)_{i=1}^N$. Then, the ensemble of systems $(\clA_{\sigma_i},B)_{i=1}^N$ is not ensemble stabilizable. Indeed, using the Hautus test we look for a vector~$v\ne0$ satisfying  $0=v\begin{bmatrix}\bfA_{\varSigma}-\lambda\Id_{nN}&\bfB\end{bmatrix}$.
Consider the case~$\lambda=0$.
For~$N\ge2$, we can choose a nonzero vector~$w\in\bbR^{1\times N}$ satisfying~${\textstyle\sum_{i=1}^N} w_{i}=0$. Next, we choose~$v$ as~$v_{3i-j}=w_{i}$ for all~$1\le i\le N$ and all~$0\le j\le 2$. In particular there holds~$v\bfB=0$. Since~$\bfA_\varSigma$ is block diagonal, the row vector $z=v\bfA_\varSigma$ has entries satisfying
 \begin{equation}\notag
 \begin{bmatrix}z_{3i-2}& z_{3i-1}&z_{3i}\end{bmatrix}=
w_{i} \begin{bmatrix}1& 1&1\end{bmatrix}\clA_{\sigma_i}=w_{i} \begin{bmatrix}0& 0&0\end{bmatrix}, \qquad\mbox{for all}\quad 1\le i\le N.
\end{equation}
It follows that~$\bfA_{\varSigma}$ is singular with~$\lambda=0$ an eigenvalue of~$\bfA_{\varSigma}$, $0=v\begin{bmatrix}\bfA_{\varSigma}&\bfB\end{bmatrix}$. 
Hence, by the Hautus test, the ensemble $(\clA_{\sigma_i},B)_{i=1}^N$ is not ensemble stabilizable.

Finally, for a single parameter,  we can see (due to the Kalman condition) that~$(\clA_\sigma,B)$ is controllable (and hence stabilizable) if, and only if, $b_{12}b_{23} \ne 0$.

 \subsubsection{A Catenary Model with Open Uncertain Free Dynamics}\label{sS:MComp-cat-open}
 We consider
  \begin{equation}\notag
\clA_{\sigma_i}= \begin{bmatrix}-b_{12}&b_{21}&0\\b_{12}&-(b_{21}+b_{23})&b_{32}\\0&b_{23}&-b_{32}-\sigma_i\end{bmatrix}, \qquad B = \begin{bmatrix} 1\\ 0 \\ 0 \end{bmatrix},
\end{equation}
with $b_{ij}>0$ for $1\le i,j\le N$, and apply the Hautus test to show ensemble controllability of $(\clA_{\sigma_i},B)_{i=1}^N$. For any given row vector~$v=[ v_1\;v_2\;\dots \;v_{3N}]\in\bbC^{1\times 3N}$, analogously to~\eqref{Hautus2},
we have $0=v\begin{bmatrix}\bfA_{\varSigma}-\lambda\Id_{nN}&\bfB\end{bmatrix}$ if, and only if,
\begin{subequations}\label{HautusCM-cat}
\begin{align}
&-(b_{12}+\lambda) v_{3i-2}+ b_{12}v_{3i-1}=0,&\;\mbox{for all } 1\le i\le N;\label{HautusCM-cat1}\\
&b_{21}v_{3i-2}-(b_{21}+b_{23}+\lambda)v_{3i-1}+b_{23}v_{3i}=0,&\;\mbox{for all }1\le i\le N;\label{HautusCM-cat2}\\
&b_{32}v_{3i-1}-(b_{32}+\lambda+\sigma_i)v_{3i}=0,&\;\mbox{for all }1\le i\le N;\label{HautusCM-cat3}\\
&{\textstyle\sum_{i=1}^N} v_{3i-2}=0.\label{HautusCM-cat4}
\end{align}
\end{subequations}
We assume again that the elements~$\sigma_i$, $1\le i\le N$, are pairwise distinct.
From~\eqref{HautusCM-cat4} and~\eqref{HautusCM-cat1}, we find~$\sum_{i=1}^N v_{3i-1}=0$ and, from~\eqref{HautusCM-cat2}, ~$\sum_{i=1}^N v_{3i}=0$. Hence,
  \begin{equation}\label{HautusCM-cat-sigma0}
{\textstyle\sum_{i=1}^N} v_{3i-j}=0,\quad\mbox{for all}\quad j\in\{0,1,2\}.
\end{equation}
By multiplying~\eqref{HautusCM-cat1} by~$(b_{21}+b_{23}+\lambda)/b_{12}$ and adding the product to~\eqref{HautusCM-cat2}, we find
  \begin{equation}\label{HautusCM-cat5}
\left(b_{21}-\frac{b_{21}+b_{32}+\lambda}{b_{12}}(b_{12}+\lambda)\right)v_{3i-2}+b_{23}v_{3i}=0,\;\mbox{for all }1\le i\le N.
\end{equation}
Now we consider two cases:
\begin{itemize}
\item[-] If~$b_{21}=\frac{b_{21}+b_{32}+\lambda}{b_{12}}(b_{12}+\lambda)$, then by~\eqref{HautusCM-cat5}, we find~$v_{3i}=0$, for all~$1\le i\le N$.  Then, by~\eqref{HautusCM-cat3} we obtain~$v_{3i-1}=0$, for all~$1\le i\le N$ and then also~$v_{3i-2}=0$, for all~$1\le i\le N$, due to~\eqref{HautusCM-cat2}. That is, ~$v=0$.
\item[-] If~$b_{21}\neq\frac{b_{21}+b_{32}+\lambda}{b_{12}}(b_{12}+\lambda)$, we proceed as follows. By~\eqref{HautusCM-cat-sigma0}, with~$j=0$ and~$j=1$, and~\eqref{HautusCM-cat3} we find~$\sum_{i=1}^N \sigma_iv_{3i}=0$. Then~${\textstyle\sum_{i=1}^N} \sigma_iv_{3i-2}=0$, due to~\eqref{HautusCM-cat5}, and also~${\textstyle\sum_{i=1}^N} \sigma_iv_{3i-1}=0$, due to~\eqref{HautusCM-cat1}. In summary we have
 \begin{equation}\label{HautusCM-cat-sigma1}
{\textstyle\sum_{i=1}^N} \sigma_iv_{3i-j}=0,\quad\mbox{for all}\quad j\in\{0,1,2\}.
\end{equation}
We next reason by induction over $k$, using the arguments above, to show that
 \begin{equation}\label{HautusCM-cat-sigmak}
{\textstyle\sum_{i=1}^N} \sigma_i^kv_{3i-j}=0,\mbox{ for all } (j,k)\in\{0,1,2\}\times\{0,1,\dots,N-1\}.
\end{equation}
Indeed, assume that for a given~$\overline k$, with~$1\le \overline k<N-1$, we have~${\textstyle\sum_{i=1}^N} \sigma_n^{\overline k}v_{3i-j}=0$, for all $j\in\{0,1,2\}$.
Then, by~\eqref{HautusCM-cat3} we obtain~$\sum_{i=1}^N\sigma_i^{\overline k+1}v_{3i}=0$ and then also~$\sum_{i=1}^N\sigma_n^{\overline k+1}v_{3i-2}=0$, due to~\eqref{HautusCM-cat5} and, consequently,
$\sum_{i=1}^N\sigma_i^{\overline k+1}v_{3i-1}=0$, due to~\eqref{HautusCM-cat1}. That is, ~${\textstyle\sum_{i=1}^N} \sigma_n^{\overline k+1}v_{3i-j}=0$, for all $j\in\{0,1,2\}$. By induction, we conclude that~\eqref{HautusCM-cat-sigmak} holds true. Therefore, we have
\begin{align}
w^{[j]}\clV_\varSigma=0,\quad\mbox{ for all}\quad j\in\{0,1,2\},\notag
\end{align}
where~$w^{[j]}\in\bbC^{1\times N}$ is the row vector with entries~$w^{[j]}_{(1,i)}=v_{3i-j}$ and~$\clV_\varSigma$ is the Vandermonde matrix in~\eqref{VanderSigma}. Since the elements~$\sigma_i$ are pairwise distinct, we conclude that~$w^{[j]}=0$, for all~$j\in\{0,1,2\}$, which implies that~$v=0$.
\end{itemize}

In either case we have~$v=0$, which means that~$\rank\begin{bmatrix}\bfA_{\varSigma}-\lambda\Id_{nN}&\bfB\end{bmatrix}=3N$, for all~$\lambda$. Hence, by the Hautus test, we have that~$(\clA_{\sigma_i},B)_{i=1}^N$ is ensemble controllable.

 \subsubsection{A Cyclic Model with Open Uncertain Free Dynamics}\label{sS:MComp-cyc-open}
 We consider
  \begin{equation}\label{sys-CM-cyc}
\clA_{\sigma_i}= \begin{bmatrix}-b_{12}&0&\sigma_i\\b_{12}&-b_{23}&b_{32}\\0&b_{23}&-b_{32}\end{bmatrix},\qquad B=\begin{bmatrix}0\\1\\0\end{bmatrix},
\end{equation}
with $b_{ij}>0$ for $1\le i,j\le N$, and with~$\sigma_i\ne0$. From the discussion following~\eqref{sys-comp-gen} we have that~\eqref{sys-CM-cyc} is a multi-compartment model only if $\sigma_i \ge 0$. Moreover, by writing~$-b_{32}=-(b_{32}+\sigma_i-\sigma_i)$, it is open iff~$\sigma_i > 0$. We again apply the Hautus test to show ensemble controllability of~$(\clA_{\sigma_i},B)_{i=1}^N$. For any given row vector~$v=[ v_1\;v_2\;\dots \;v_{3N}]\in\bbC^{1\times 3N}$, analogously to~\eqref{Hautus2}, we have $0=v\begin{bmatrix}\bfA_{\varSigma}-\lambda\Id_{nN}&\bfB\end{bmatrix}$ if, and only if,
\begin{subequations}\label{HautusCM-cyc}
\begin{align}
&-(b_{12}+\lambda) v_{3i-2}+b_{12}v_{3i-1}=0,&\;\mbox{for all } 1\le i\le N;\label{HautusCM-cyc1}\\
&-(b_{23}+\lambda)v_{3i-1}+b_{32}v_{3i}=0,&\;\mbox{for all }1\le i\le N;\label{HautusCM-cyc2}\\
&\sigma_iv_{3i-2}+b_{23}v_{3i-1}-(b_{32}+\lambda)v_{3i}=0,&\;\mbox{for all }1\le i\le N;\label{HautusCM-cyc3}\\
&{\textstyle\sum_{i=1}^N} v_{3i-1}=0.\label{HautusCM-cyc4}
\end{align}
\end{subequations}
We assume again that the elements~$\sigma_i$, $1\le i\le N$, are pairwise distinct.
Now, we consider two cases:
\begin{itemize}
\item[-] Case~$\lambda\ne-b_{12}$. From~\eqref{HautusCM-cyc4} and~\eqref{HautusCM-cyc1}, we find~$\sum_{i=1}^N v_{3i-2}=0$. Further, from~\eqref{HautusCM-cyc4} and~\eqref{HautusCM-cyc2} we find~$\sum_{i=1}^N v_{3i}=0$. Hence,
  \begin{equation}\label{HautusCM-cat-sigma02}
{\textstyle\sum_{i=1}^N} v_{3i-j}=0,\quad\mbox{for all}\quad j\in\{0,1,2\}.
\end{equation}
By induction, we argue that
 \begin{equation}\label{HautusCM-cyc-sigmak}
{\textstyle\sum_{i=1}^N} \sigma_i^kv_{3i-j}=0,\mbox{ for all } (j,k)\in\{0,1,2\}\times\{0,1,\dots,N-1\}.
\end{equation}
Indeed, from~\eqref{HautusCM-cat-sigma02} we know that~\eqref{HautusCM-cyc-sigmak} holds for~$k=\overline k=1$. Assume now that ${\textstyle\sum_{i=1}^N} \sigma_i^{\overline k}v_{3i-j}=0$, for all~$j\in\{0,1,2\}$.
Then, by~\eqref{HautusCM-cyc3}, $\sum_{i=1}^N \sigma_i^{\overline k+1}v_{3i-2}=0$, by~\eqref{HautusCM-cyc1}, $\sum_{i=1}^N \sigma_i^{\overline k+1}v_{3i-1}=0$, and by~\eqref{HautusCM-cyc2}, $\sum_{i=1}^N \sigma_i^{\overline k+1}v_{3i-1}=0$. That is, ${\textstyle\sum_{i=1}^N} \sigma_i^{\overline k+1}v_{3i-j}=0$, for all~$j\in\{0,1,2\}$.
Thus, we conclude that~\eqref{HautusCM-cyc-sigmak} holds true. Therefore, we have
\begin{align}
w^{[j]}\clV_\varSigma=0,\quad\mbox{ for all}\quad j\in\{0,1,2\},\notag
\end{align}
where~$w^{[j]}\in\bbC^{1\times N}$ is the row vector with entries~$w^{[j]}_{(1,i)}=v_{3i-j}$ and~$\clV_\varSigma$ as in~\eqref{VanderSigma}. Since the elements~$\sigma_i$ are pairwise distinct, we conclude that~$w^{[j]}=0$, for all~$j\in\{0,1,2\}$, which implies that~$v=0$.
\item[-] Case~$\lambda=-b_{12}$. By~\eqref{HautusCM-cyc1} we have that~$v_{3i-1}=0$ for all~$1\le i\le N$. Then, by~\eqref{HautusCM-cyc2} it follows~$v_{3i} =0$ for all~$1\le i\le N$ and, by~\eqref{HautusCM-cyc3}, it follows that~$\sigma_iv_{3i-2}=0$ for all~$1\le i\le N$. Which implies that~$v_{3i-2}=0$ for all~$1\le i\le N$, since~$\sigma_i\ne0$.
\end{itemize}
Therefore, we conclude that~$(\clA_{\sigma_i},B)_{i=1}^N$ is ensemble controllable.

\subsection{Spectral Heat Equation}
Given a finite ensemble of parameters $\varSigma = (\sigma_i)_{i=1}^N$, we consider the parameterized heat equation, with state~$y=y(t,s)$, defined for time~$t>0$ and for~$s$ in the spatial interval/domain~$(0,2\pi)$,
\begin{subequations}\label{eq:spec_heat}
\begin{align}
 &\frac{\p}{\p t} y(t,s) = \frac{\p^2}{\p s^2} y(t,s) + \sigma_i y(t,s) + \Id_{\omega}(s) u(t),\qquad
 y(0,s) = y_{\circ}, \label{eq:spec_heat1}\\
 &y(t,0) = 0=y(t,2\pi), \label{eq:spec_heat2}
\end{align}
\end{subequations}
where~\eqref{eq:spec_heat2}
specifies the Dirichlet boundary conditions. Our actuator is the indicator function~$\Id_{\omega}$ of the subinterval~$\omega \coloneqq (\omega_1,\omega_2) \subseteq (0,2\pi)$. Here, the uncertain parameter is the reaction coefficient~$\sigma_i$.

We recall that, the eigenvalues of the (negative) Laplacian~$- \frac{\p^2}{\p s^2}$ are given by $\lambda_k = \frac{k^2}{4}$, for integers~$k\ge1$, corresponding to a complete system of eigenfunctions $\{\varphi_k(s) = \sin(\frac{k}{2}s)\mid k\ge1\}$. For each parameter~$\sigma_i$, a (spectral) Galerkin discretization of the solution $y(t) = \sum_{k=1}^{+\infty} y_k(t) \sin(\frac{k}{2}s)$, can be found as~$y^x(t,s) \coloneqq \sum_{k=1}^M x_k(t) \sin(\frac{k}{2}s)$, where $x(t) = \begin{bmatrix} x_1(t) & x_2(t) & \ldots & x_M(t) \end{bmatrix}^\top$ solves 
\begin{subequations}\label{eq:SpecHeat}
\begin{align}
 \dot{x}&= \clA_{\sigma_i} x + B u,\qquad x(0) = x_{\circ},
\end{align}
with the diagonal~$\clA_{\sigma_i}$ and column~$B$ matrices, with entries~$(\clA_{\sigma_i})_{(j,k)}$ and~$(B)_{(j,1)}$, in the $j$-th row and~$k$-th column, as follows,
\begin{align}
&(\clA_{\sigma_i})_{(j,j)} \coloneqq\sigma_i-\lambda_j\quad\mbox{and}\quad  (\clA_{\sigma_i})_{(j,k)} \coloneqq0, &&\quad\mbox{for}\quad 1\le j\ne k\le M,\label{eq:SpecHeatA}\\
 &(B)_{(j,1)} \coloneqq\frac{2}{j\pi}\Bigl(\cos(\tfrac{j}2\omega_1) - \cos(\tfrac{j}2\omega_2)\Bigr), &&\quad\mbox{for}\quad 1\le j\le M,\label{eq:SpecHeatB}
\end{align}
\end{subequations}
and with~$x_\circ\coloneqq [y_{\circ 1}\;y_{\circ 2}\;\dots\;y_{\circ n}]\in\bbR^n$, where~$y_\circ(s) \eqqcolon \sum_{k=1}^{+\infty}y_{\circ k} \sin(\frac{k}{2}s)$. Note that
~$\Id_{\omega}=\sum_{j=1}^{+\infty} (B)_{(j,1)}\varphi_j$, with~$(B)_{(j,1)}$ as in~\eqref{eq:SpecHeatB} for~$j\in\bbN$.
\begin{lemma}\label{lem:SpecHeat}
If  for all $1\le k< j\le M$ we have that~$(B)_{(j,1)} \ne 0$ and~$4(\sigma_k - \sigma_j)$ is not an integer, then the ensemble $(\clA_{\sigma_i},B)_{i=1}^N$ in \eqref{eq:SpecHeat} is ensemble controllable.
\end{lemma}
\begin{proof}
If for all $1\le k< j\le M$, $4(\sigma_k - \sigma_j)$ is not an integer, then the eigenvalues~$\beta=\sigma_i-\lambda_j$ of~$\bfA_{\varSigma}$, $1\le i\le N$, $1\le j\le M$, are pairwise distinct. In this case $\rank\begin{bmatrix} \bfA_{\varSigma} - \beta \Id_{MN} \end{bmatrix} = MN-1$ for all $\beta \in {\rm Eig}(\bfA_{\varSigma})$. Replacing the only vanishing column of the diagonal matrix~$ \bfA_{\varSigma} - \beta \Id_{MN}$ by~$\bfB$ (recall~\eqref{eq:extsys-op}) we obtain a nonsingular matrix. Note that~$(\bfB)_{(i,1)}\ne0$ for all $1\le  i\le MN$, since~$(B)_{(j,1)} \ne 0$ for all $1\le  j\le M$. Thus, we can conclude that~$\rank \begin{bmatrix}
    \bfA -  \beta \Id_{MN} & \bfB
   \end{bmatrix} = MN
$ for all $ \beta \in  {\rm Eig}(\bfA_{\varSigma})$ and, consequently, that~$(\clA_{\sigma_i},B)_{i=1}^N$ is ensemble controllable, by the Hautus test~\eqref{eq:HautusContr}. \hspace{-0.2cm} \end{proof}

\section{Numerical Experiments}\label{sec:numEx}
We present numerical experiments supporting our theoretical findings for the examples discussed in Section~\ref{sec:enstabsys}. More precisely, for given parameter ensembles $\varSigma = (\sigma_i)_{i=1}^N$ and linear systems~\eqref{eq:sys}, we show the stabilizing performance of the feedback $-\frac1{\alpha}\bfB^\top \bfPi_{\varSigma}\clE$ introduced in~\eqref{eq:linFB}, with $\bfPi_{\varSigma}\succ0$ solving the  Riccati equation \eqref{eq:bigRiccati}. This performance is compared to that of the optimal  feedback~$-\frac1{\alpha}B^\top\Pi_{\bar{\sigma}}$ corresponding to the mean~$\bar{\sigma} \coloneqq \frac1N \sum_{i=1}^N \sigma_i$ of $\varSigma$, where~$\Pi_{\bar\sigma}\succ0$ is the solution of the Riccati equation~\eqref{eq:smallRiccati} (with~$\clA_{\bar\sigma}$ in place of~$\clA_{\sigma}$).

Note that we can write~$\bfB=\clE B$ and~$-\frac1{\alpha}\bfB^\top\bfPi_{\varSigma}\clE=-\frac1{\alpha}B^\top\clE^\top\bfPi_{\varSigma}\clE$. Hence we shall compare the two feedback control operators as follows,
\begin{equation}\label{preFeed}
-\frac1{\alpha}B^\top \Pi\quad\mbox{with}\quad\Pi\in\left\{\clE^\top\bfPi_{\varSigma}\clE,\Pi_{\bar{\sigma}}\right\},\qquad\bar{\sigma} \coloneqq \frac1N \sum_{i=1}^N \sigma_i.
\end{equation}

We shall refer to either of the operators $\Pi$ in~\eqref{preFeed} as feedback operator.

Recalling the notation in~\eqref{eq:fbcontrol}, for a given parameter~$\sigma\in\bbR$, the state-control pair associated with the feedback~$\clE^\top\bfPi_{\varSigma}\clE$ is~$(x_{\varSigma,\sigma},u_{\varSigma,\sigma})$. Similarly, we denote the state-control pair associated with the feedback~$\Pi_{\bar\sigma}$ by~$(x_{\bar\sigma,\sigma},u_{\bar\sigma,\sigma})$.

In concrete examples hereafter, it is convenient to denote the ensemble of parameters as an ordered set by~$\varSigma=(\sigma_i)_{i=1}^N\eqqcolon\{\sigma_1; \sigma_2;\dots;\sigma_N\}$

\subsection{The Oscillator}\label{sec:NumOsc}
Consider~\eqref{sys-oscu} with $x_{\circ} = \begin{bmatrix} 1 & 0 \end{bmatrix}^\top$. In the Riccati equations~\eqref{eq:bigRiccati} and~\eqref{eq:smallRiccati} we take~$\alpha = 0.1$ and run the simulations for time~$t\in(0,T)$ with~$T = 5$. We start by comparing the pairs~$(x_{\varSigma,\sigma},u_{\varSigma,\sigma})$  and~$(x_{\bar\sigma,\sigma},u_{\bar\sigma,\sigma})$ for the ensemble 
\begin{equation}\label{osc.Sig1}
\varSigma=(\sigma_i)_{i=1}^5\coloneqq \{-0.5;-0.25;0;0.25;0.5\}.
\end{equation}
The time evolution of the controls and norms of the corresponding states are displayed in Fig.~\ref{fig:4}, for the cases~$\sigma\in\varSigma$.  
\begin{figure}[htb]
    \centering
    \includegraphics[width=.85\textwidth]{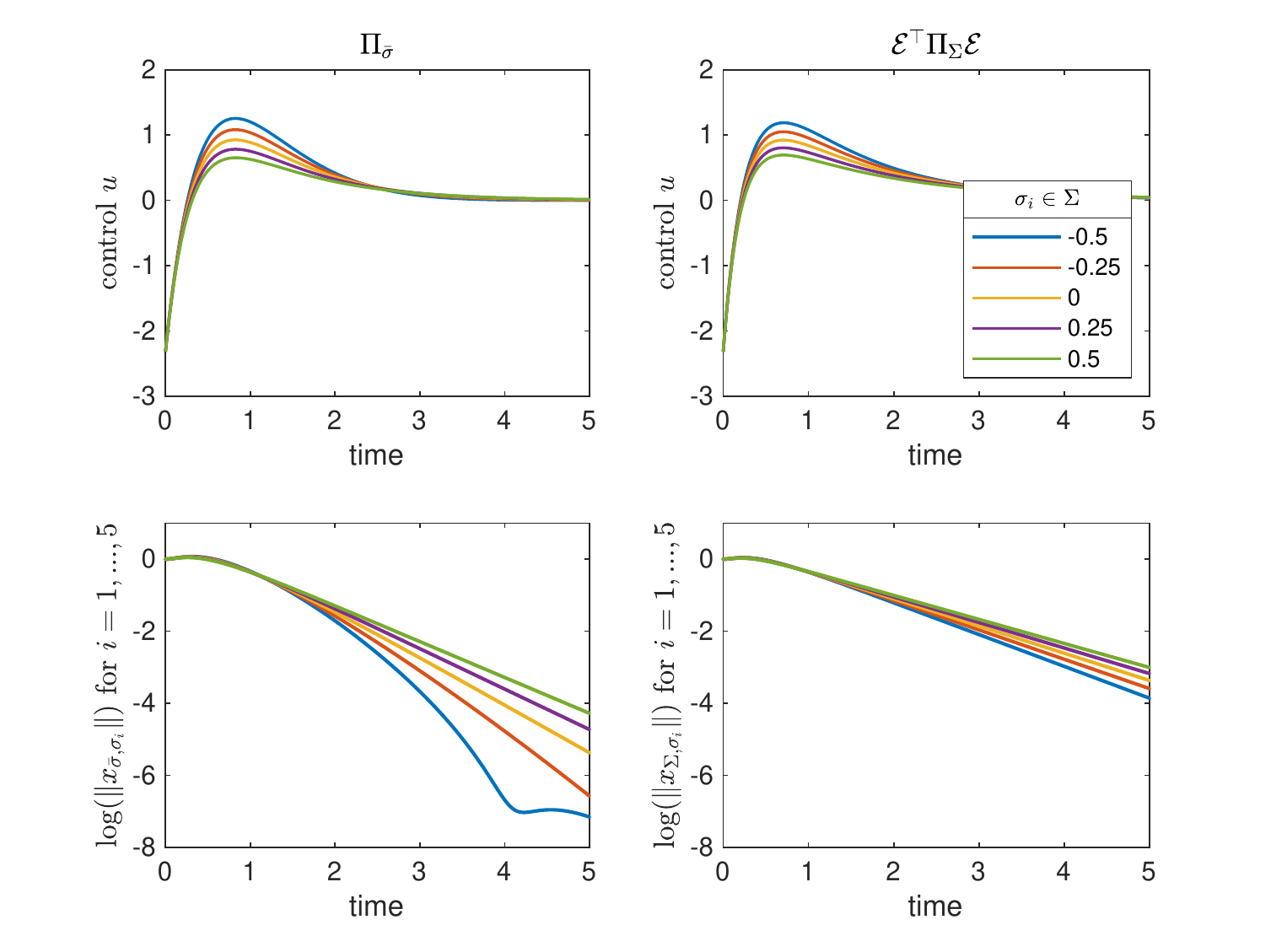}
    \caption{Time evolution of control and norm of state, for $\sigma_i \in\varSigma$ as in~\eqref{osc.Sig1}. Left: feedback~$\Pi_{\bar{\sigma}}$. Right: feedback~$\clE^\top\bfPi_{\varSigma}\clE$.}
    \label{fig:4}
\end{figure}
\begin{table}[htb]
\centering
\renewcommand{\arraystretch}{1.5}
\begin{tabular}{|c|l|c|c|}
\hline
\multirow{2}{6em}{Control cost} & $\frac15 \sum_{i=1}^5\int_0^T \frac{\alpha}{2} (u_{\bar{\sigma},\sigma_i}(t))^\top u_{\bar{\sigma},\sigma_i}(t)\,\rmd t$ & 0.0673 \\
\cline{2-3}
& $\frac15 \sum_{i=1}^5 \int_0^T\frac{\alpha}{2}(u_{\varSigma,\sigma_i})^\top u_{\varSigma,\sigma_i}(t)\,\rmd t$ & 0.0672\\
\hline
\multirow{2}{6em}{State cost} & $\frac15 \sum_{i=1}^5\int_0^T \frac12 (x_{\bar{\sigma},\sigma_i})^\top x_{\bar{\sigma},\sigma_i}(t)\,\rmd t$& 0.5732\\
\cline{2-3}
& $\frac15 \sum_{i=1}^5\int_0^T \frac12 (x_{\varSigma,\sigma_i})^\top x_{\varSigma,\sigma_i}(t)\,\rmd t$ & 0.5982\\
\hline
\end{tabular}
\vskip 10pt
\caption{Comparison of the cost for $\sigma_i \in\varSigma$ as in~\eqref{osc.Sig1}.}
  \label{tab:1}
\end{table}
Averages of control and state components of the associated costs (restricted to the time interval~$(0,5)$) are collected in Table~\ref{tab:1}.
Comparing the controls, we see that the controls $u_{\bar{\sigma},\sigma_i}$ have marginally higher costs (on average) while steering the states faster to zero, which leads  to lower state costs (on average) for $x_{\bar{\sigma},\sigma_i}$.

Next, we increase the range of the parameters, which corresponds to more uncertainty. The results are shown in Fig.~\ref{fig:5} with the associated costs compared in Table~\ref{tab:2},  for
\begin{equation}\label{osc.Sig2}
\varSigma=(\sigma_i)_{i=1}^5\coloneqq \{-4;-2;0;2;4\}.
\end{equation}
\begin{figure}[htb]
    \centering
    \includegraphics[width=.85\textwidth]{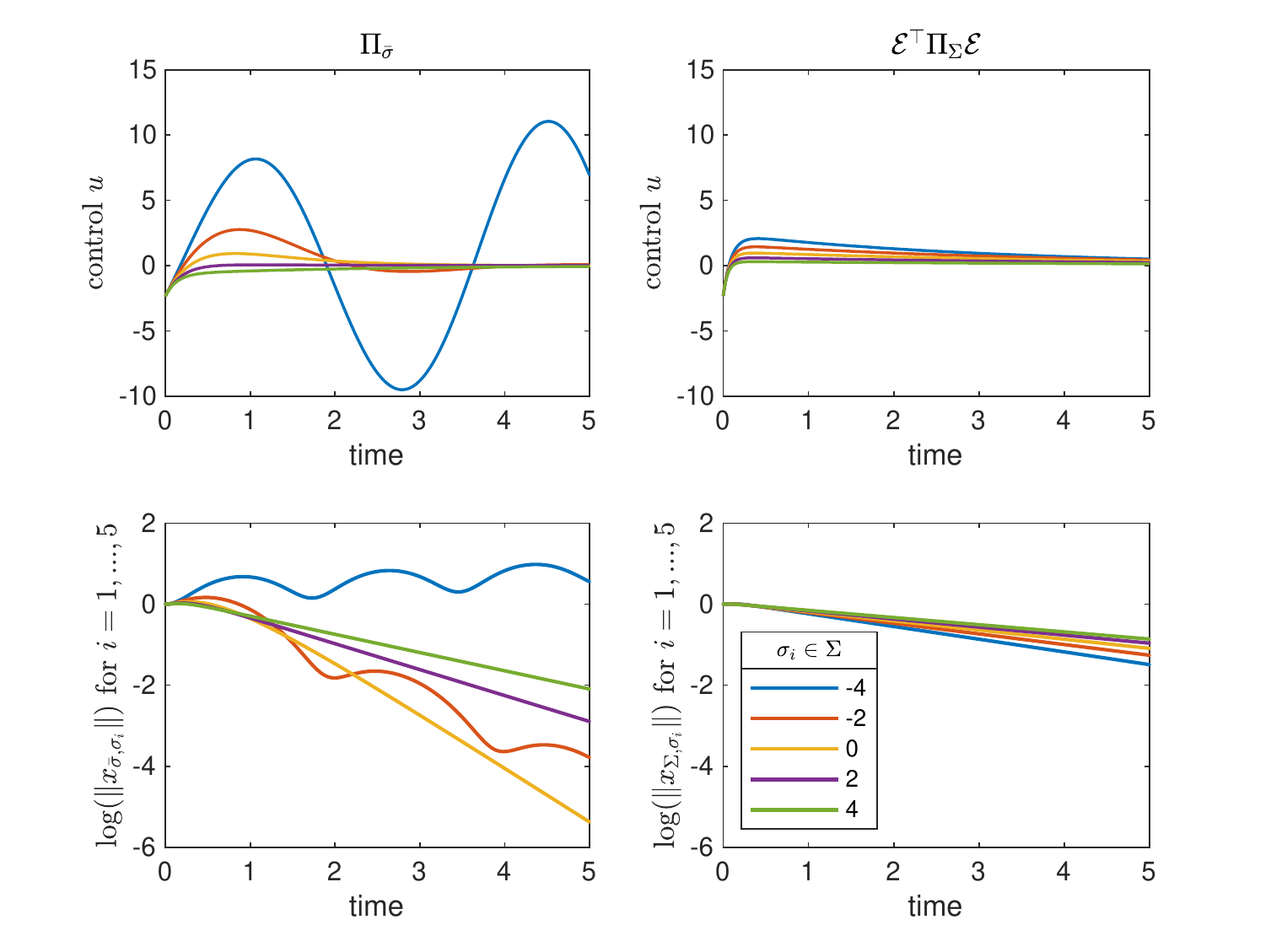}
    \caption{Time evolution of control and norm of state, for $\sigma_i \in\varSigma$ as in~\eqref{osc.Sig2}. Left: feedback~$\Pi_{\bar{\sigma}}$. Right: feedback~$\clE^\top\bfPi_{\varSigma}\clE$.} \label{fig:5}
\end{figure}
\begin{table}[htb]
\centering
\renewcommand{\arraystretch}{1.5}
\begin{tabular}{|c|l|c|c|}
\hline
\multirow{2}{6em}{Control cost} & $\frac15 \sum_{i=1}^5\int_0^T\frac{\alpha}{2} (u_{\bar{\sigma},\sigma_i}(t))^\top u_{\bar{\sigma},\sigma_i}(t)\,\rmd t$ & 2.4505\\
\cline{2-3}
& $\frac15 \sum_{i=1}^5 \int_0^T\frac{\alpha}{2}(u_{\varSigma,\sigma_i}(t))^\top u_{\varSigma,\sigma_i}(t)\,\rmd t$ & 0.1570\\
\hline
\multirow{2}{6em}{State cost} & $\frac15 \sum_{i=1}^5\int_0^T \frac{1}{2} (x_{\bar{\sigma},\sigma_i}(t))^\top x_{\bar{\sigma},\sigma_i}(t)\,\rmd t$& 2.3027\\
\cline{2-3}
& $\frac15 \sum_{i=1}^5\int_0^T \frac{1}{2} (x_{\varSigma,\sigma_i}(t))^\top x_{\varSigma,\sigma_i}(t)\,\rmd t$ & 1.0620\\
\hline
\end{tabular}
\vskip 10pt
\caption{Comparison of the cost, for $\sigma_i \in\varSigma$ as in~\eqref{osc.Sig2}.}
  \label{tab:2}
\end{table}

In these two experiments we observe that in situations with little uncertainty about the parameters, the feedback~$\Pi_{\bar{\sigma}}$ is an effective choice as it requires solving the lower-dimensional algebraic Riccati equation~\eqref{eq:smallRiccati}, instead of~\eqref{eq:bigRiccati}. However, if the uncertainty (i.e., the parameter range) is large, the feedback $\Pi_{\bar{\sigma}}$ may fail to stabilize the system, whereas the feedback $\clE^\top \bfPi_{\varSigma} \clE$ appears to be more robust.
In particular, we see that~$\Pi_{\bar{\sigma}}$ fails to stabilize the system for $\sigma_1 = -4$, whereas~$\clE^\top \bfPi_{\varSigma} \clE$ is stabilizing for all $\sigma_i \in \{-4;-2;0;2;4\}$.

In these experiments, and following ones, we solved the algebraic Riccati equation using the \texttt{matlab} function \texttt{icare}. In the above examples, we obtained residuals as
\begin{align}
&\|\bfA_{\varSigma}^\top{\bfPi_{\varSigma} + \bfPi_{\varSigma} \bfA_{\varSigma} }- \frac{1}{\alpha}{\bfPi_{\varSigma} \bfB \bfB^\top \bfPi_{\varSigma}} + \frac{1}{N}\Id_{nN} \|_{\max} \approx 10^{-12};\notag\\
&\|\clA_{\bar\sigma}^\top{\Pi_{\bar\sigma} + \Pi_{\bar\sigma} \clA_{\bar\sigma} }- \frac{1}{\alpha}{\Pi_{\bar\sigma}B B^\top \Pi_{\bar\sigma}} +\Id_{n} \|_{\max} \approx  10^{-15},\notag
\end{align}
with~$\|A\|_{\max} \coloneqq \max_{i,j}|a_{i,j}|$.

\subsubsection{On the Eigenvalues of the Closed-loop System}
Here, we consider also the feedback~$\Pi$ given by the mean of the optimal Riccati feedbacks associated to each~$\sigma_i\in\varSigma$, leading us to the  feedback control operator
\begin{equation}\label{preFeed3}
-\frac1{\alpha}B^\top \Pi\quad\mbox{with}\quad\Pi\coloneqq \frac1N \sum_{i=1}^N \Pi_{\sigma_i}.
\end{equation}
Further, we consider several values for~$\alpha$ in the Riccati equations~\eqref{eq:bigRiccati} and~\eqref{eq:smallRiccati}. On the left side in Fig.~\ref{fig:10},
\begin{figure}[htb]
    \centering
    \includegraphics[width=.85\textwidth]{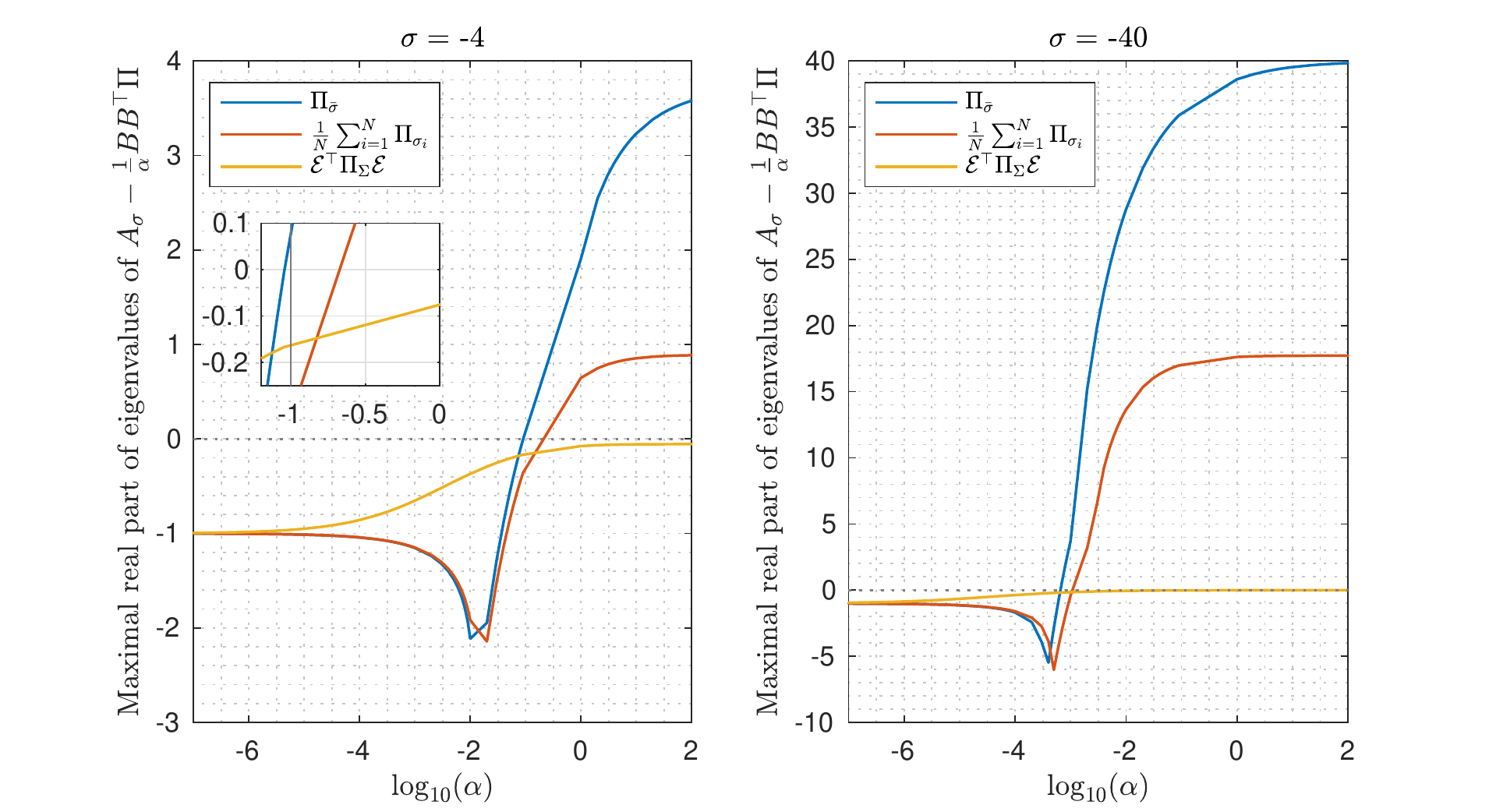}
    \caption{Maximal real part of the eigenvalues of $\clA_{\sigma} - \frac1\alpha BB^\top \Pi$ for
    feedbacks~$\Pi \in \{\clE^\top\bfPi_{\varSigma}\clE,\Pi_{\bar{\sigma}},\frac1N \sum_{i=1}^N \Pi_{\sigma_i}\}$. Left: $\sigma=-4$ and~$\varSigma =\{-4;-2;0;2;4\}$. Right: $\sigma=-40$ and~$\varSigma = \{-40;-20;0;20;40\}$.    
}
    \label{fig:10}
\end{figure} 
we see that the feedback $\Pi_{\bar{\sigma}}$, for $\alpha = 10^{-1}$, does not stabilize the system for $\sigma_i = -4$, whereas the feedback $\clE^\top \bfPi_{\varSigma} \clE$ does; see also Fig.~\ref{fig:5}. Moreover, in this example the feedback $\clE^\top \bfPi_\varSigma \clE$ stabilizes the system for all plotted values of $\alpha >0$. Further, in Fig.~\ref{fig:10} we see that, when increasing the parameter range, in this example, the maximal real part of the eigenvalues of $\clA_{\sigma} - \frac1\alpha BB^\top \Pi$, with~$\Pi\in\{\Pi_{\bar{\sigma}}, \frac1N \sum_{i=1}^N \Pi_{\sigma_i}\}$ change sign for smaller values of $\alpha$.

\subsubsection{On the Operator Norm of~$\bfPi_\varSigma$.}
We recall that some results in Section~\ref{sec:linfb} (see, e.g.,~\eqref{eq:smallness} within Corollary~\ref{coro:stableFeedback}) 
depend on the constants~$\beta_1,\beta_2$ in~\eqref{eq:Pinorm} and, hence, also on $\|\bfPi_\varSigma\|$. Therefore, we report on experiments concerning the dependence of $\|\bfPi_{\varSigma}\|$ on the ensemble $\varSigma$.
Let us denote a uniform $N$-elements partition of a bounded interval~$[a,b]\subset\bbR$, $a<b$, as
 the sequence~$[a,b]_N\coloneqq(a+\frac{i-1}{N-1}(b-a))_{i=1}^N$.
 { Table~\ref{tab:4} highlights the fact that $\|\bfPi_{\varSigma}\|$ is smaller for parameters  associated with more stable dynamics (i.e., with larger~$\sigma$ in the present example). Moreover, $\|\bfPi_{\varSigma}\|$ increases with~$N$, see Table~\ref{tab:5}. Note that for larger~$N$, the minimal distance between two parameters in~$[a,b]_N$ decreases, hence, the increase of $\|\bfPi_{\varSigma}\|$ aligns with Lemma~\ref{lem:stab-enstab}, which implies (for our example) that stabilizability will fail if we have a repeated unstable parameter~$\sigma$ in~$\varSigma$. Recall also that, by Lemma~\ref{lem:enscontr}, $(\clA_{\sigma_i},B)_{i=1}^N$ is ensemble controllable if the parameters are pairwise distinct.  Finally, in Table~\ref{tab:6}, we fix the number of parameters to be $N=3$, while we increase the interval over which they are distributed.   As expected, the smallest and the largest intervals lead to larger values of $\|\bfPi_{\varSigma}\|$ than the midsize intervals.

\begin{table}[htb]
\centering
\renewcommand{\arraystretch}{1.5}
\begin{tabular}{|c|c|c|c|}
\hline
 $\varSigma$  & $[-1.5,-0.5]_3$ & $[-0.5,0.5]_3$ & $[0.5,1.5]_3$ \\
\cline{1-4}
$\beta_2$ & $36.2087$ & $3.3834$ & $0.8357$\\
\hline
$\beta_1$ & $0.3800$ & $0.3258$ & $0.2786$ \\
\hline
\end{tabular}
\vskip 10pt
\caption{Square root of the largest and smallest eigenvalues of~$\bfPi_{\varSigma}$. Shifting the set of parameters.}
  \label{tab:4}
\end{table}
\begin{table}[htb]
\centering
\renewcommand{\arraystretch}{1.5}
\begin{tabular}{|c|c|c|c|}
\hline
 $\varSigma=[-0.5,0.5]_N$ & $N=3$ & $N=5$ & $N=7$ \\
\cline{1-4}
$\beta_2$ & $3.3834$ & $12.8909$ & $56.7993$\\
\hline
$\beta_1$ & $0.3258$ & $0.2531$ & $0.2141$ \\
\hline
\end{tabular}
\vskip 10pt
\caption{Square root of the largest and smallest eigenvalues of~$\bfPi_{\varSigma}$. Finer partition of an interval.}
  \label{tab:5}
\end{table}
\begin{table}[htb]
\centering
\renewcommand{\arraystretch}{1.5}
\begin{tabular}{|c|c|c|c|c|}
\hline
 $\varSigma=[-0.5L,0.5L]_3$ & $L=1$ & $L=2$& $L=3$ & $L=10$ \\
\cline{1-5}
$\beta_2$ & $3.3834$ & $2.8728$ & $2.7616$ & $3.2346$\\
\hline
$\beta_1$ & $0.3258$ & $0.3141$ & $0.2949$ & $0.1824$ \\
\hline
\end{tabular}
\vskip 10pt
\caption{Square root of the largest and smallest eigenvalues of~$\bfPi_{\varSigma}$. Length of  parameter interval.}
  \label{tab:6}
\end{table}

\begin{remark}\label{rem:nnecessary}
Let us comment on the applicability of condition~\eqref{eq:smallness} in Corollary~\ref{coro:stableFeedback} to the present example. While from Section~\ref{sec:NumOsc} we know that the feedback~$\bfclE^\top \bfPi_\varSigma \bfclE$ stabilizes the systems for the ensemble~$\varSigma = \{-0.5;-0.25;0;0.25;0.5\}$, Table~\ref{tab:5} reveals that~\eqref{eq:smallness} is not fullfilled. The condition is thus sufficient but not necessary. We next report on a specific example where~\eqref{eq:smallness} is applicable. For~$N$ equidistant parameters~$\varSigma = (\sigma_i)_{i=1}^N$ in the interval 
$I \coloneqq [1.475,1.525]$, condition~\eqref{eq:smallness} holds true for all~$N$ up to at least~$256$ and all parameters $\sigma \in I$.
\end{remark}

\subsubsection{Remark on Finite and Infinite Time-horizon Problems}

We have already commented  on the relationship between the open-loop and 
the closed-loop  presentations \eqref{eq:extcontr} and 
\eqref{eq:optcondextE} of the optimal control.

In the finite time-horizon (FTH) case (i.e., with $\infty$ in \eqref{eq:extobj} replaced by $T>0$), the open-loop ensemble optimal control problem has been investigated for instance in \cite{guth2022parabolic}. The FTH optimal feedback control operator is given by~$-\frac1{\alpha}\bfB^\top\mathbf\Pi_{\varSigma}^T$, where~$\mathbf\Pi_{\varSigma}^T=\mathbf\Pi_{\varSigma}^T(t)$ solves, for time~$t\in(0,T)$, the differential Riccati equation
\begin{align}\label{eq:DRE}
-\dot{\mathbf\Pi}_{\varSigma}^{T} = \bfA_{\varSigma}^\top 
\mathbf\Pi_{\varSigma}^T + \mathbf\Pi_{\varSigma}^T 
\bfA_{\varSigma} - \frac{1}{\alpha} \mathbf\Pi_{\varSigma}^T 
\bfB \bfB^\top \mathbf\Pi_{\varSigma}^T + \frac{1}{ N} 
\Id_{nN},\qquad\mathbf{\Pi}_{\varSigma}^T(T) = \mathbf{0}.
\end{align}
Recall that in the infinite time-horizon (ITH) case the optimal feedback control operator is given by~$-\frac1{\alpha}\bfB^\top\mathbf\Pi_{\varSigma}$ where~$\mathbf\Pi_{\varSigma}$ solves the algebraic  Riccati equation~\eqref{eq:bigRiccati}.  

We next give a numerical example, with $\varSigma=\{-0.5;-0.25;0;0.25;0.5\}$  to demonstrate, firstly, that (up to numerical error) the computed open-loop and closed-loop ensemble optimal controls do indeed coincide and, secondly, that such a comparison is more delicate for the ITH problem, since computing for the ITH inevitably necessitates an additional approximation step. 
Note also that the free dynamics of~\eqref{sys-oscu} is not stable for $\sigma \le 0$.
For the difference between the open-loop solution $(u_{\rm OL},\bfx_{\rm OL})$ and the solution $(u_{\varSigma}^T,\bfx_{\varSigma}^T)$ corresponding to the differential Riccati equation we obtain
$ \|u_{\rm OL} - u_{\varSigma}^T\|_{L^2((0,T);\bbR)} = 1.2107 \cdot 10^{-5}$ and
$ \|\bfx_{\rm OL} - \bfx_{\varSigma}^T\|_{L^2((0,T);\bbR^{10})} = 1.5282 \cdot 10^{-5}.$
Due to these small values, we see no difference between the optimal open-loop and the closed-loop controls as depicted in Fig.~\ref{fig:outlook} (left). Here  the 
open-loop problem is solved using a gradient method with Barzilai--Borwein steps, see~\cite{azmi2020analysis}. The state and adjoint differential 
equations in the gradient steps as well as the differential Riccati 
equation \eqref{eq:DRE} are solved using the \texttt{matlab} function 
\texttt{ode45}.

Concerning the asymptotic behavior as $T\to \infty$, we know that $\lim_{T\to \infty} \mathbf{\Pi}_{\varSigma}^T(0) = \bfPi_{\varSigma}$ (see, e.g.,~\cite[Thm. 2.3.9.1]{lasiecka2000control_2}). However, depending on the structure of the underlying 
system, we may need large~$T$ to obtain~$\mathbf{\Pi}_{\varSigma}^T(0)$ accurately approximating~$\bfPi_{\varSigma}$, as illustrated in Fig.~\ref{fig:outlook} (right).
\begin{remark}
The value~$\bfPi_\varSigma^T(0)$ for the solution of~\eqref{eq:DRE}, at time~$t=0$, coincides with~$\widehat\bfPi(-T)$ where $\widehat\bfPi(t)$ solves the dynamics in~\eqref{eq:DRE} for time~$t\in(-\infty,0)$ with final condition~$\widehat\bfPi(0)=0$, at time~$t=0$. Thus, in fact, we can compute~$\widehat\bfPi(-T)$ instead.  
\end{remark}
\begin{figure}[htb]
    \centering
    \includegraphics[width=.85\textwidth]{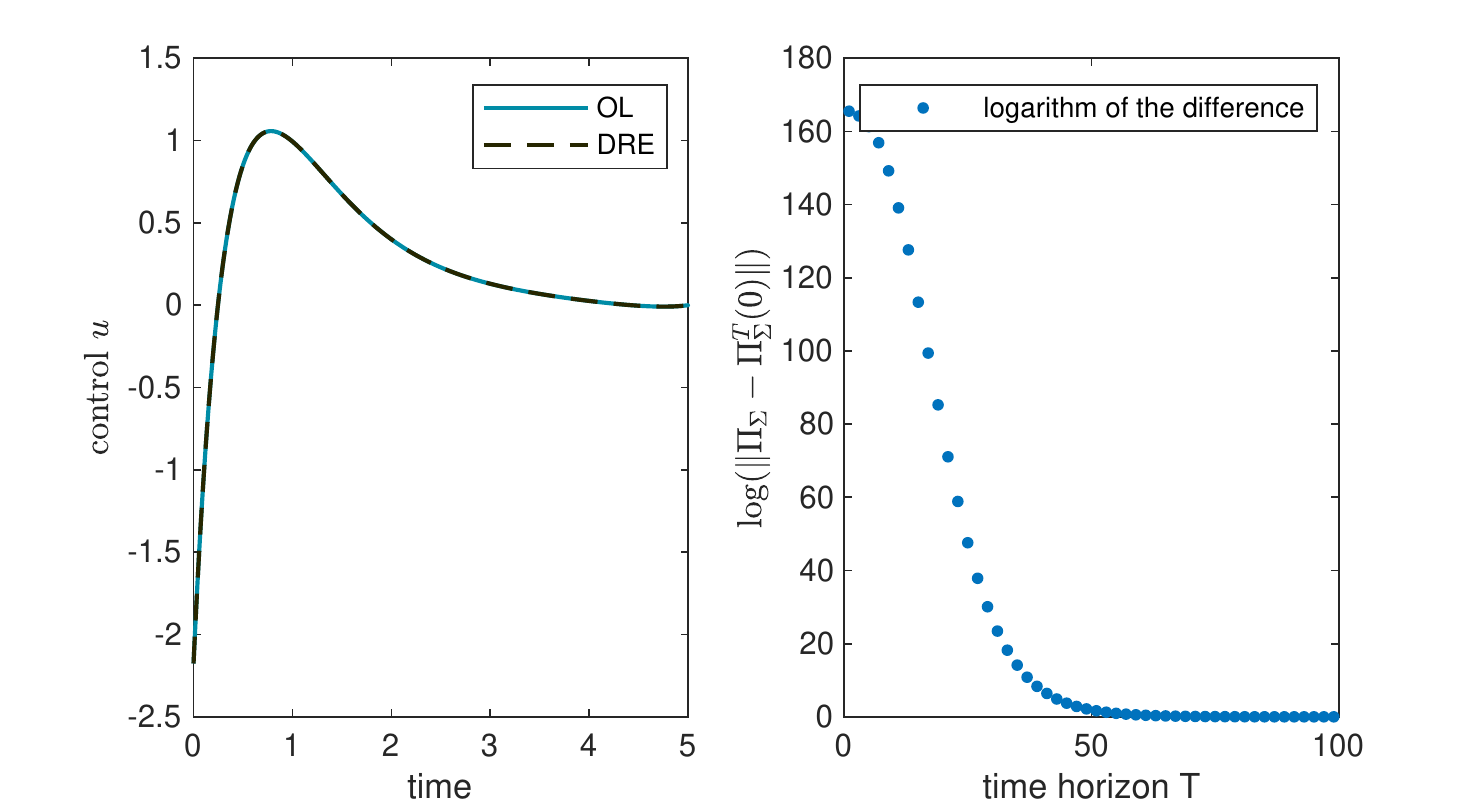}
    \caption{Left: FTH optimal controls by solving the open-loop first order optimality conditions (OL) and the differential Riccati equation~\eqref{eq:DRE} (DRE); $T=5$. Right: operator norm of the difference between the solution of~\eqref{eq:bigRiccati} and the solution of~\eqref{eq:DRE} evaluated at~$t=0$.}
    \label{fig:outlook}
\end{figure}

\subsection{A Cyclic Model with Uncertain Free Dynamics}
Consider the ensemble of systems $(\clA_{\sigma_i},B)_{i=1}^5$ given in \eqref{sys-CM-cyc} with $(b_{12},b_{23},b_{32})=(1,2,1)$, $x_{\circ} = \begin{bmatrix} 1 & 1 & -1 \end{bmatrix}^\top$ and ensemble of parameters
\begin{equation}\label{compSig1}
\varSigma=(\sigma_i)_{i=1}^5=\{-1000;-500;0;500;1000\}.
\end{equation}
 In Fig.~\ref{fig:6-1}, we see that the feedback $\Pi_{\bar{\sigma}}$ stabilizes the system for  the parameter~$\sigma=\sigma_3 = 0 = \bar{\sigma}$, but fails to stabilize the system for the remaining parameters $\sigma_i$, $i\ne3$, whereas the feedback $\clE^\top \bfPi_{\varSigma} \clE$ is stabilizing for all the parameters in~$\varSigma$.

\begin{figure}[htb]
    \centering
    \includegraphics[width=.85\textwidth]{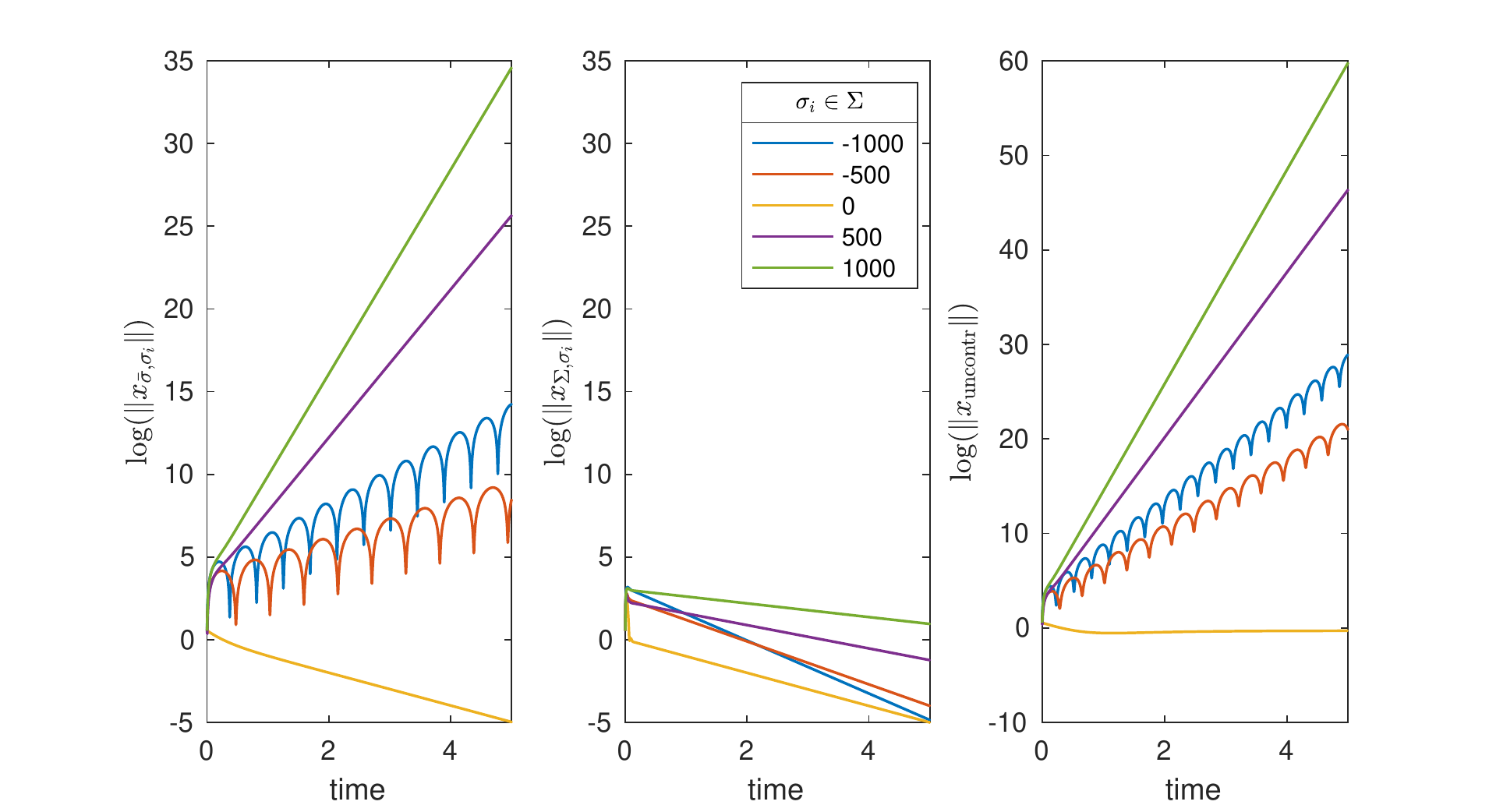}
    \caption{Evolution of norm of the states. $\varSigma$ as in~\eqref{compSig1}. Left:  feedback $\Pi_{\bar{\sigma}}$. Middle:  feedback $\clE^\top \bfPi_{\varSigma} \clE$. Right: free dynamics.}
    \label{fig:6-1}
\end{figure}
\begin{table}[ht]
\centering
\renewcommand{\arraystretch}{1.5}
\begin{tabular}{|c|l|c|c|}
\hline
\multirow{2}{6em}{Control cost} & $\frac15 \sum_{i=1}^5\int_0^T\frac{\alpha}{2} (u_{\bar{\sigma},\sigma_i}(t))^\top u_{\bar{\sigma},\sigma_i}\,\rmd t$ & $5.8303\cdot 10^{26}$ \\
\cline{2-3}
& $\frac15 \sum_{i=1}^5\int_0^T \frac{\alpha}{2}(u_{\varSigma,\sigma_i}(t))^\top u_{\varSigma,\sigma_i}\,\rmd t$ & 1186.3073 \\
\hline
\multirow{2}{6em}{State cost} & $\frac15 \sum_{i=1}^5\int_0^T \frac12 (x_{\bar{\sigma},\sigma_i}(t))^\top x_{\bar{\sigma},\sigma_i}\,\rmd t$& $9.1831\cdot 10^{27}$\\
\cline{2-3}
& $\frac15 \sum_{i=1}^5\int_0^T \frac12 (x_{\varSigma,\sigma_i}(t))^\top x_{\varSigma,\sigma_i}(t)\,\rmd t$ & 81.5806\\
\hline
\end{tabular}
\vskip 10pt
\caption{Comparison of the cost for $\sigma_i \in \varSigma$ as in~\eqref{compSig1}.}
  \label{tab:3}
\end{table}

Moreover, in Fig.~\ref{fig:6-2}, we see that~$\clE^\top \bfPi_{\varSigma} \clE$ also stabilizes the systems $(\clA_{\sigma_i},B)_{i=1}^N$ for new parameters $\sigma \notin \varSigma$, whereas the former feedback $\Pi_{\bar{\sigma}}$ fails to stabilize the system for any of the new parameters.

\begin{figure}[htb]
    \centering
    \includegraphics[width=.85\textwidth]{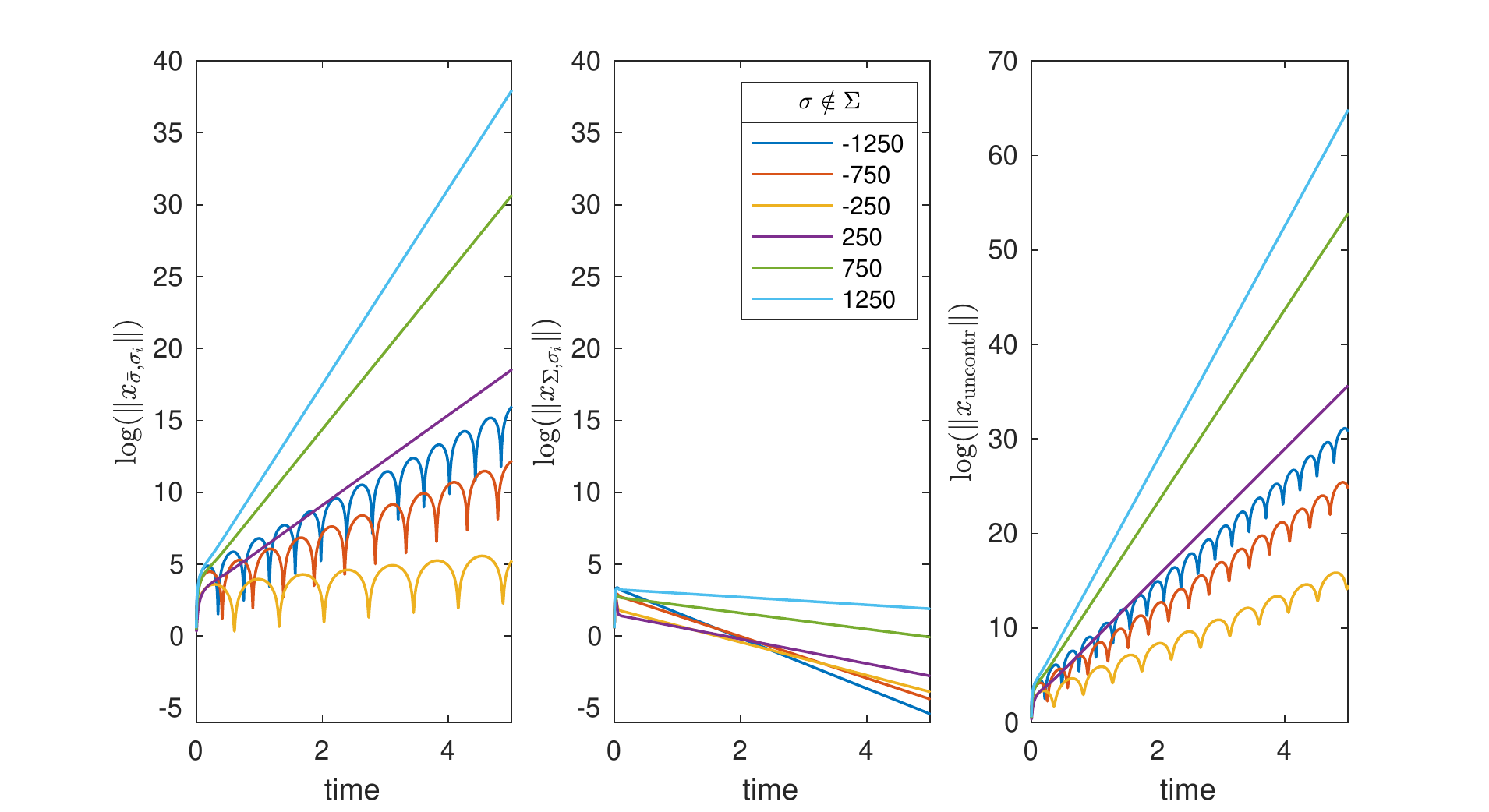}
    \caption{Evolution of norm of the states. $\varSigma$ as in~\eqref{compSig1}. Left: feedback $\Pi_{\bar{\sigma}}$. Middle:  feedback $\clE^\top \bfPi_{\varSigma} \clE$. Right: free dynamics.}
    \label{fig:6-2}
\end{figure}

Finally, we give an example where the proposed feedback $\clE^\top \bfPi_{\varSigma} \clE$ fails to stabilize the system for a parameter $\sigma\in\varSigma$. For this purpose we take the ensemble
\begin{equation}\label{compSig2}
\varSigma=(\sigma_i)_{i=1}^5=\{-1000;0;1000\}.
\end{equation}
\begin{figure}[htb]
    \centering
    \includegraphics[width=.85\textwidth]{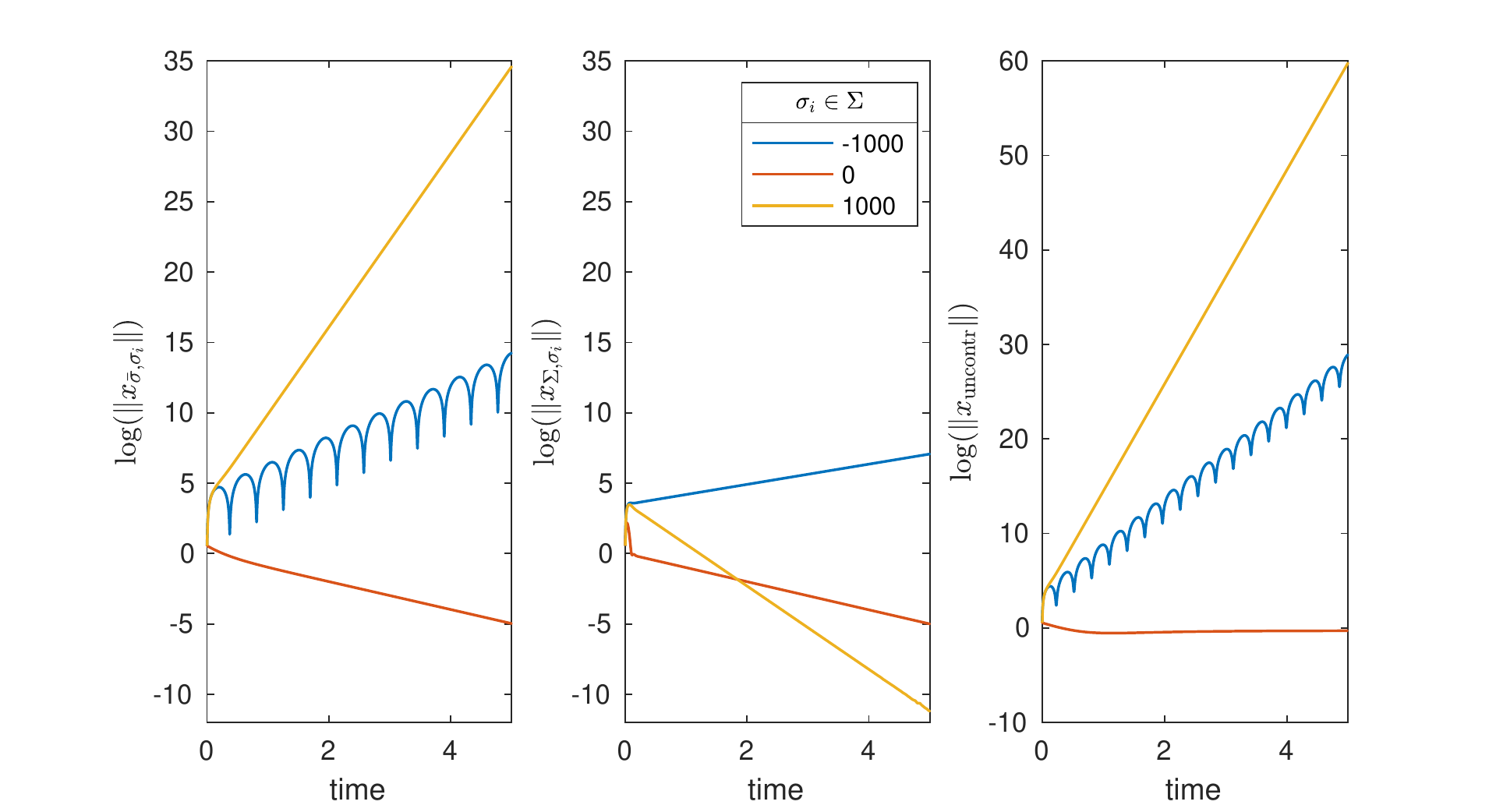}
    \caption{Evolution of norm of the states. $\varSigma$ as in~\eqref{compSig2}. Left: feedback $\Pi_{\bar{\sigma}}$. Middle: feedback $\clE^\top \bfPi_{\varSigma} \clE$. Right: free dynamics.}
    \label{fig:6}
\end{figure}
In Fig.~\ref{fig:6}, we see that the feedback $\clE^\top \bfPi_{\varSigma} \clE$ fails to stabilize the system for $\sigma_3=1000$. Still, the same figure also suggests that it is more robust than the feedback $\Pi_{\bar{\sigma}}$, since the latter is stabilizing only for~$\sigma=\sigma_2=0=\bar{\sigma}$.

Comparing the negative results obtained for $\varSigma$ as in~\eqref{compSig2} to the positive ones obtained for  $\varSigma$ as in~\eqref{compSig1}, we infer that the ensemble~$\varSigma$ of training parameters must be rich/fine enough to obtain a robust stabilizing feedback.

\subsection{Spectral Heat Equation} Consider the systems $(\clA_{\sigma_i},B)_{i=1}^3$ given in \eqref{eq:SpecHeat} with $M=5$, $x_{\circ} = \begin{bmatrix} -1 & -1 & -1 & -1 & -1 \end{bmatrix}^\top$, 
$\omega=(\omega_1,\omega_2) =(1,2)$, and
\begin{equation}\label{heatSig1}
\varSigma=(\sigma_i)_{i=1}^3=\Bigl\{-\frac{7}{12}\pi;-\frac{2}{12}\pi;\frac{1}{4}\pi\Bigr\}.
\end{equation}
Then, $(\clA_{\sigma_i},B)_{i=1}^3$ is ensemble controllable by Lemma~\ref{lem:SpecHeat}. Indeed, using the notation in~\eqref{eq:SpecHeatA}-\eqref{eq:SpecHeatB}, we find that
$(B)_{(j,1)} =\frac{2}{j\pi}( \cos(\frac{j}{2}) - \cos(j)) =\frac{2}{j\pi}(\cos(\frac{j}{2}) - 2\cos^2(\frac{j}{2}) +1)$, where we have used~$\cos(j) =2\cos^2(\frac{j}{2}) -1$.
Thus, ~$(B)_{(j,1)}=0$ if, and only if, $\cos(\frac{j}{2}) = 1$ or $\cos(\frac{j}{2}) = -0.5$, that is,
\begin{align}
(B)_{(j,1)}=0\quad\Longleftrightarrow\quad \frac{j}2\in \Bigl\{0,\frac23\pi,\frac43\pi\Bigr\}+2\pi\bbZ,
\end{align}
where~$\bbZ$ is the set of integer numbers.
Since $j$ is a positive integer, we conclude that $(B)_{(j,1)}\neq 0$.
Further, we have that, for~$k<j$,~$4(\sigma_k-\sigma_j)\in\{- \frac53\pi,-\frac{10}3\pi\}$ is not an integer. Thus, by Lemma~\ref{lem:SpecHeat}, $(\clA_{\sigma_i},B)_{i=1}^3$ is ensemble controllable.

Note that~$\clA_{\sigma_3}$ has exactly one nonnegative eigenvalue, namely~$\sigma_3-\lambda_1=\frac{\pi}4-\frac14$, whereas~$\clA_{\sigma_1}$ and~$\clA_{\sigma_2}$ have no nonnegative eigenvalues.

Similarly to the previous experiments, the feedback $\clE^\top\bfPi_{\varSigma}\clE$ appears to be more robust than $\Pi_{\bar{\sigma}}$. The latter fails to stabilize the system~$(\clA_{\sigma},B)$ for~$\sigma=\sigma_3\in\varSigma$, whereas the former stabilizes~$(\clA_{\sigma},B)$ for all~$\sigma\in\varSigma$.

\begin{figure}[htb]
    \centering
    \includegraphics[width=.85\textwidth]{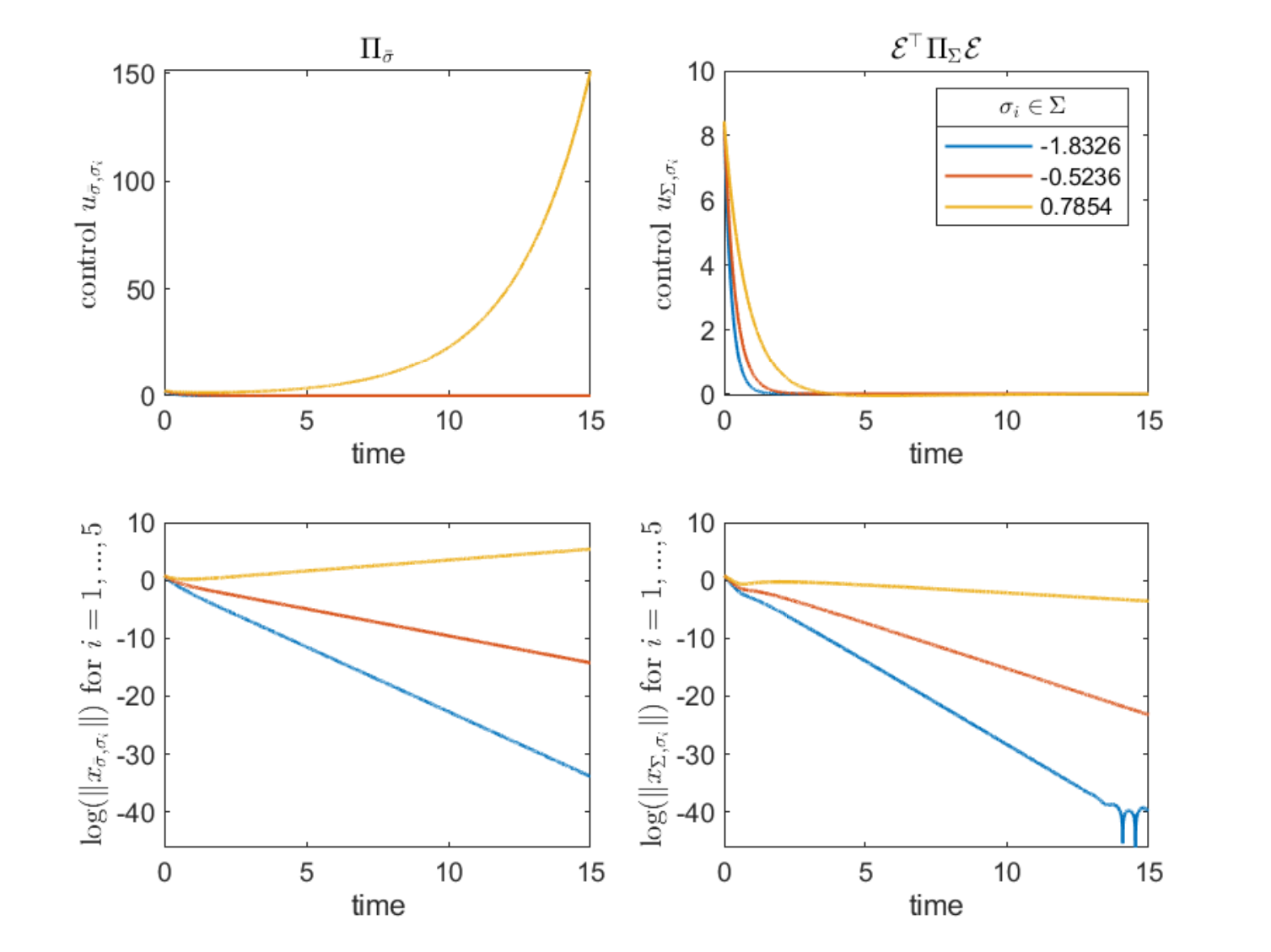}
    \caption{Time evolution of control and norm of state, for $\sigma_i \in\varSigma$ as in~\eqref{heatSig1}. Left: feedback~$\Pi=\Pi_{\bar{\sigma}}$. Right: feedback~$\Pi=\clE^\top\bfPi_{\varSigma}\clE$.}
    \label{fig:7}
\end{figure}

\section{Conclusions}
This manuscript is concerned with feedback stabilization of ensembles of finitely many linear dynamical systems. Based on the algebraic Riccati equation for a suitable extended system depending on the ensemble of (training) parameters, we develop a linear feedback law and provide conditions under which the obtained feedback control stabilizes each system in the ensemble. We emphasize that the same feedback operator is used for any given realization of the uncertain parameter.
Under appropriate assumptions we prove results on the associated (optimal) costs and on the stabilizing performance of the proposed feedback. We illustrate our theoretical results for a set of application-motivated examples and confirm our findings in numerical experiments, showing interesting stabilizing and robustness properties of the proposed feedback.

\bigskip\noindent
\textbf{Acknowledgements.} S. Rodrigues gratefully acknowledges partial support from
the State of Upper Austria and Austrian Science
Fund (FWF): P 33432-NBL.

\bibliography{EnsembleFeedback}
\bibliographystyle{plainurl}

\end{document}